\documentclass{article}
\usepackage{subfigure}
\usepackage{authblk}
\usepackage{cellspace}%
\usepackage{makecell}
\setcellgapes{3pt}

\usepackage{fullpage,enumerate,comment}

\makeatletter
\def\keywords{\xdef\@thefnmark{}\@footnotetext}
\makeatother

\usepackage{tikz}

\usetikzlibrary{matrix} 
\usetikzlibrary{arrows,automata}
\usetikzlibrary{positioning}

\usepackage{amssymb,mathrsfs,amsmath,amsthm,bm,diagbox,float}
\usepackage{multirow}
\usetikzlibrary{matrix,arrows}

\usepackage{hyperref}

\usepackage{mathtools}

\renewcommand{\S}{\mathcal{S}}

\newcommand{\T}{\mathcal{T}}

\DeclareMathOperator{\cyc}{cyc}
\DeclareMathOperator{\lrmax}{lrmax}
\DeclareMathOperator{\red}{red}
\DeclareMathOperator{\des}{des}

\newtheorem{theorem}{Theorem}[section]

\newtheorem{proposition}[theorem]{Proposition}
\newtheorem{corollary}[theorem]{Corollary}
\newtheorem{lemma}[theorem]{Lemma}
\newtheorem{conjecture}[theorem]{Conjecture}
\theoremstyle{definition}
\newtheorem{example}[theorem]{Example}
\newtheorem{remark}[theorem]{Remark}
\newtheorem{question}[theorem]{Question}

\title{Pattern-avoiding shallow permutations}
\author[1]{Kassie Archer}
\author[1]{Aaron Geary}
\author[1]{Robert P. Laudone}
\affil[1]{{\small Department of Mathematics, United States Naval Academy, Annapolis, MD, 21402}}
\affil[ ]{{\small Email: \{karcher, geary, laudone\}@usna.edu }}

\date{}
\begin{document}

\maketitle

\begin{abstract}
Shallow permutations were defined in 1977 to be those that satisfy the lower bound of the Diaconis-Graham inequality. Recently, there has been renewed interest in these permutations. In particular, Berman and Tenner showed they satisfy certain pattern avoidance conditions in their cycle form and Woo showed they are exactly those whose cycle diagrams are unlinked. 
Shallow permutations that avoid 321 have appeared in many contexts; they are those permutations for which depth equals the reflection length, they have unimodal cycles, and they have been called Boolean permutations. 
Motivated by this interest in 321-avoiding shallow permutations, we investigate $\sigma$-avoiding shallow permutations for all $\sigma \in \S_3$. To do this, we develop more general structural results about shallow permutations, and apply them to enumerate shallow permutations avoiding any pattern of length 3. 

\end{abstract}


\section{Introduction and background}


Let $\S_n$ denote the set of permutations on $[n]=\{1,2,\ldots,n\}$ and we write these permutations in their one-line notation as $\pi = \pi_1\pi_2\ldots\pi_n$ where $\pi_i:=\pi(i)$. There are multiple measures of disorder or disarray of a permutation $\pi \in \S_n$. Three of these, namely the total displacement $D(\pi)$, the length $I(\pi)$, and the reflection length $T(\pi)$, are connected by the Diaconis-Graham inequalities \cite{DG77}:
\[
I(\pi) + T(\pi) \leq D(\pi) \leq 2 \, I(\pi).
\]
Here, the total displacement $D(\pi)$, also called Spearman's measure of disarray, is given by 
\[
D(\pi) = \sum_{i = 1}^n |\pi_i - i|.
\]
The length $I(\pi)$ is equal to the minimal number of simple transpositions required to produce $\pi$. It is also called the inversion number and is given by
\[
I(\pi) = \sum_{i=1}^n |\{i < j \; | \; \pi_i > \pi_j\}|.
\]
The reflection length $T(\pi)$ is the minimal number of transpositions required to produce $\pi$ from the identity permutation,  which was shown by Cayley in 1849 to be
\[
T(\pi) = n - \cyc(\pi),
\]
where $\cyc(\pi)$ denotes the number of cycles in the disjoint cycle decomposition of $\pi.$ 

It is well-known that the upper Diaconis-Graham inequality is achieved, i.e., $D(\pi)=2I(\pi)$, when $\pi$ avoids the pattern 321, meaning there is no set of indices $i<j<k$ with $\pi_i>\pi_j>\pi_k$. A permutation is called {\em shallow} when it satisfies the lower inequality, i.e., when $I(\pi) + T(\pi) = D(\pi)$. We note that shallow permutations have recently been investigated from various perspectives: In \cite{BT22}, the authors use pattern functions to characterize the cycle form of these permutations in terms of pattern-avoidance, and in \cite{W22}, the author proves that shallow permutations are exactly those permutations whose cycle diagram is equivalent to the unknot when viewed as a knot diagram. 

Permutations which satisfy both the upper and lower bound of the Diaconis-Graham inequalities have been well-studied in their own right. These permutations are exactly those that are shallow 321-avoiding permutations; these have been called Boolean permutations \cite{RT1,RT2,T07}, unimodal permutations \cite{E87} (because of their unimodal cycle form), and are characterized as avoiding both 321 and 3412 \cite{PT15}.
 It was stated, without proof, in \cite{DG77}, that these permutations are enumerated by $F_{2n-1}$, the $(2n-1)$-st Fibonacci number. A proof of this fact does appear in other places, including \cite{PT15}, and we provide an independent proof of this fact in this paper, directly using shallowness.
 
Motivated by this interesting answer regarding 321-avoiding shallow permutations, in this paper we investigate shallow permutations which avoid $\sigma$ for $\sigma \in \S_3$. In Section \ref{sec:shallow}, we describe certain properties of general shallow permutations which we use in follow-on sections. In particular, we show how to build shallow permutations from smaller ones, and we prove that all shallow permutations necessarily avoid certain mesh patterns. In Sections \ref{sec:132}, \ref{sec:231}, \ref{sec:123}, and \ref{sec:321} we enumerate $\sigma$-avoiding shallow permutations for $\sigma \in \S_3$. Additionally, we enumerate $\sigma$-avoiding shallow permutations by number of descents and by three symmetry properties. In particular, we enumerate those shallow $\sigma$-avoiding permutations that are fixed under inverse, reverse-complement, and reverse-complement-inverse. The sections are ordered by the complexity of the proofs involved, with the exception of $\sigma=321$, which we do last since these have been investigated in previous papers. We conclude the paper with open questions and directions for future study. 

\renewcommand{\arraystretch}{2.2}
\begin{table}[ht]
    \centering
    \begin{tabular}{c|c|c}
       $\sigma$  & number of shallow $\sigma$-avoiding permutations & Theorem  \\ \hline\hline
        $132$ & \multirow{ 2}{*}{$t_n(\sigma)=F_{2n-1}$}& \multirow{2}{*}{Theorem~\ref{thm: 132}}\\ \cline{1-1}
       $213$ & & \\ \hline
       $231$ & \multirow{ 2}{*}{g.f. $T_{\sigma}(x) = \dfrac{1-3x+2x^2-x^3-x^4-x^5}{1-4x+4x^2-2x^3-x^4-x^5}$}& \multirow{2}{*}{Theorem~\ref{thm:231}} \\ \cline{1-1}
       $312$ & & \\ \hline
       $123$  & g.f. $T_{\sigma}(x) =\dfrac{1-3x+11x^3-13x^4+7x^5+6x^6+3x^7}{(1-x)^4 (1 - 4 x^2 + x^4)}$ & Theorem~\ref{thm: 123}\\ \hline
       $321$ & $t_n(\sigma)=F_{2n-1}$ & Theorem~\ref{thm: 321} \\
     \end{tabular}
    \caption{In this table, $t_n(\sigma)$ denotes the number of shallow permutations avoiding a given pattern $\sigma$, and $T_{\sigma}(x) = \sum_{n\geq 0} t_n(\sigma)x^n$ is the corresponding generating function.}
    \label{tab:my_label}
\end{table}

\section{Structure of shallow permutations}\label{sec:shallow}

Let $\mathcal{T}_n$ denote the permutations $\pi \in \S_n$ that are shallow and let $t_n = |\mathcal{T}_n|$. We will often make use of the following recursive formulation of shallow permutations that is due to Hadjicostas and Monico \cite{HM13}.

\begin{theorem}\cite[Theorem 4.1]{HM13} \label{thm: shallowRecursive}
Suppose $\pi \in \mathcal{S}_n$ and for $n\geq 2$, define 
\[
\pi^R = \begin{cases}
    \pi_1\pi_2\ldots\pi_{n-1} & \pi_n=n \\
    \pi_1\pi_2\ldots\pi_{j-1}\pi_n\pi_{j+1}\ldots \pi_{n-1} & \pi_j=n \text{ with } j<n
    \end{cases}
\]
Then $\pi =1 \in \S_1$ is shallow, and when $n\geq 2,$
\begin{itemize}
    \item if $\pi_n = n$, then $\pi$ is shallow exactly when $\pi^R$ is, and
    \item if $\pi_j=n$ with $j<n$,  then $\pi$ is shallow exactly when $\pi^R$ is shallow and $\pi^R_j = \pi_n$ is a left-to-right maximum or right-to-left minimum in $\pi^R$.
\end{itemize}  
\end{theorem}


Let us see an example. Suppose $\pi= 421635\in \T_6$, a shallow permutation of 6 elements. Notice $\pi_4=6$, and $\pi_6=5$; applying the $\pi^R$ map we see
\[
421\underline{\mathbf{6}}3\underline{\mathbf{5}} \xlongrightarrow{\pi^R} 421\underline{\mathbf{5}}3,
\]
and 42153 is $\in \T_5$. Notice that we can use the inverse of this map to construct new shallow permutations from old ones. Given any $\tau\in\T_{n-1}$ and any position $i$ for which $\tau_i$ is either a left-to-right maximum or a right-to-left minimum, we can construct a permutation $\pi$ for which $\pi^R=\tau$ by taking $\pi_j=\tau_j$ for $j\neq i$, $\pi_i=n$, and $\pi_n=\tau_i.$ Notice that we can get every shallow permutation on $[n]$ from the shallow permutations on $[n-1]$ in this way since every shallow permutation $\pi\in\T_n$ has an image $\tau=\pi^R$ in $\T_{n-1}.$


We will define a similar operation $\pi^L$ which acts on the left of $\pi$. To this end, let us define certain symmetries of permutations in general. We also give the names of permutations which are fixed under these symmetries, which will be relevant throughout this paper. 
\begin{itemize}
    \item We denote by $\pi^{-1}$ the algebraic inverse of $\pi.$ That is, $\pi^{-1}_j=i$ if and only if $\pi_i=j$. This corresponds to a reflection of the diagram of the permutation (given by points $(j,\pi_j)$ for each $j\in[n]$) about the main diagonal. Permutations which are their own inverse are called \emph{involutions}.
    \item We define $\pi^{rc}$ to be the reverse-complement of $\pi$, so that $\pi^{rc}_{n+1-i} = n+1-j$ if $\pi_i=j$. This corresponds to a $180^\circ$ rotation of the diagram of the permutation. Permutations satisfying $\pi=\pi^{rc}$ are called \emph{centrosymmetric}; see for example~\cite{E07,LO10}.
    \item Finally, let $\pi^{rci}:=(\pi^{rc})^{-1}$ be the reverse-complement-inverse of the permutation, corresponding to the reflection of the diagram of the permutation about the anti-diagonal. We will refer to permutations satisfying $\pi = \pi^{rci}$ as \emph{persymmetric}.
\end{itemize}
 In the following proposition we show that each of these three symmetries preserves shallowness.

\begin{proposition} \label{prop: invRc}
    If $\pi$ is shallow, then so are the permutations $\pi^{-1}$, $\pi^{rc},$ and $\pi^{rci}$.
\end{proposition}

\begin{proof}
    To see that $\pi^{-1} \in \mathcal{T}_n$ notice first that $D(\pi^{-1}) = \sum_{i=1}^n |\pi^{-1}_{\pi_i} - \pi_i| = \sum_{i=1}^n |i - \pi_i| = D(\pi)$. Next, $I(\pi^{-1}) = I(\pi)$ since this is also the length of $\pi^{-1}$ and if $\pi = s_{i_1}\cdots s_{i_k}$ then $\pi^{-1} = s_{i_k} \cdots s_{i_1}$. Similarly, $T(\pi^{-1})= T(\pi)$ since the cycle type of $\pi$ and $\pi^{-1}$ are the same; if $p = c_1 \cdots c_\ell$ then $p^{-1} = c_1^{-1} \cdots c_\ell^{-1}$. So since $I(\pi) + T(\pi) = D(\pi)$ the same is true for $\pi^{-1}$ meaning $\pi^{-1}$ is shallow.

    We similarly check $\pi^{rc}$. First, 
    \[
    D(\pi^{rc}) = \sum_{i=1}^n |\pi^{rc}_i - i| = \sum_{i=1}^n |(n-\pi_{n-i+1}+1)-i| = \sum_{i=1}^n |(n-i+1) - \pi_{n-i+1}| = D(\pi).
    \]
    Next, $I(\pi^{rc}) = I(\pi)$ since $\pi$ has an inversion in position $(i,j)$ if and only if $\pi^{rc}$ has one in position $(n-i+1,n-j+1)$. Indeed $\pi_i > \pi_j$ with $i < j$ if and only if $\pi^{rc}_{n-i+1} = n-\pi_i +1 < n-\pi_j + 1 = \pi^{rc}_{n-j+1}$ with $n-i+1 > n-j+1$. Finally $\pi^{rc}$ and $\pi$ have the same cycle type because $\pi^{rc} = \sigma^{-1} \pi \sigma$ where $\sigma = n(n-1)(n-2)\cdots21$. Two permutations have the same cycle type if and only if they are conjugate. Finally, $\pi^{rci}$ preserves shallowness because the reverse-complement and inverse do.
\end{proof}

We can now use Proposition~\ref{prop: invRc} to define a similar operation to $\pi^R$, which we denote by $\pi^L$, which also preserves shallowness and is defined as follows. Here, the reduction operator, $\red$, takes the elements of its input of length $\ell$ and returns a permutation in $\S_\ell$ in the same relative order. For example, $\red(48291)= 34251$ and $\red(9482) = 4231.$

\begin{theorem} \label{thm: shallowRecursiveL}
Suppose $\pi \in \mathcal{S}_n$ and for $n\geq 2$, define 
\[
\pi^L = \begin{cases}
    \red(\pi_2\ldots\pi_{n-1}) & \pi_1=1 \\
    \red(\pi_2\ldots\pi_{j-1}\pi_1\pi_{j+1}\ldots \pi_{n}) & \pi_j=1 \text{ with } j>1
    \end{cases}
\]
Then $\pi =1 \in \S_1$ is shallow, and when $n\geq 2,$
\begin{itemize}
    \item if $\pi_1 = 1$, then $\pi$ is shallow exactly when $\pi^L$ is, and
    \item if $\pi_j=1$ with $j>1$,  then $\pi$ is shallow exactly when $\pi^L$ is shallow and $\pi^L_j = \pi_n$ is a left-to-right maximum or right-to-left minimum in $\pi^L$.
\end{itemize}  
\end{theorem}

\begin{proof}
    This follows immediately from Theorem \ref{thm: shallowRecursive} and Proposition \ref{prop: invRc} since $\pi^L = [(\pi^{rc})^R]^{rc}$. 
\end{proof}

Let us see an example. If $\pi= 421635\in \T_6$, we can apply $\pi^L$ to get
\[
\underline{\mathbf{4}}2\underline{\mathbf{1}}635 \xlongrightarrow{\pi^L} \text{red}(2\underline{\mathbf{4}}635)=1\underline{\mathbf{3}}524,
\]
and note that $13524\in \T_5$. Similar to our observation about the right operator above, this left operator can also be ``inverted'' to produce all shallow permutations on $[n]$ from those on $[n-1]. $

We denote by $\pi^{R^n}$ and $\pi^{L^n}$ the application the right and left operators from Theorems~\ref{thm: shallowRecursive} and \ref{thm: shallowRecursiveL}, respectively, applied $n$ times. For example, $\pi^{L^3} = ((\pi^L)^L)^L.$ On occasion, after applying the left operator to a permutation, we will work with the entries of the resulting permutation without reducing, for ease of notation. When we do this, we mark the entries. For example, we may write $(421635)^{L}$ as $2'4'6'3'5'$ with $i'=i-1$, instead of writing 13524. More generally if $\pi = \pi_1 \pi_2 \ldots \pi_{j-1} \pi_j \pi_{j+1} \ldots \pi_n$ and $\pi_j = 1$ we may refer to $\pi^L$ as $\pi^L=\pi_1' \pi_2' \ldots \pi_{j-1}' \pi_1' \pi_{j+1}' \ldots \pi_n'$ with $\pi_i'=\pi_i-1$ for each $i\neq j$ instead of writing $\pi^L=(\pi_1-1) (\pi_2-1) \ldots (\pi_{j-1}-1) (\pi_1-1) (\pi_{j+1}-1) \ldots (\pi_n-1)$. 

Next, let us make some general observations about shallowness. In the following lemma, we will see that shallowness is preserved under direct sums. Here, if $\pi \in \S_n$ and $\sigma\in\S_m$, then $\tau = \pi\oplus\sigma$ denotes the permutation in $\S_{m+n}$ with $\tau_i=\pi_i$ for all $i\in[n]$ and $\tau_{j+n}=\sigma_j$ for all $j\in[m]$. For example, we have that  $4312\oplus 53142 = 431297586.$

\begin{lemma}\label{lem: dir sum}
    If $\pi\in\T_n$ and $\sigma\in\T_m$, then $\pi\oplus\sigma\in\T_{n+m}.$
\end{lemma}

\begin{proof}
    First, notice that $D(\pi \oplus \sigma) = \sum_{i=1}^n |\pi_i - i| + \sum_{i=1}^{m} |(\sigma_{i}+n) - (i+n)|= D(\pi) +D(\sigma)$. Next, $I(\pi \oplus \sigma) = I(\pi) + I(\sigma)$ since there can be no additional inversions between the elements of $\pi$ and $\sigma$. Finally, $T(\pi \oplus \sigma) = T(\pi) + T(\sigma)$ since the number of cycles in $\pi \oplus \sigma$ is the sum of the number of cycles in $\pi$ plus the number of those in $\sigma$. It then follows from the original definition of a shallow permutation that if $\pi$ and $\sigma$ are shallow, so is $\pi \oplus \sigma$.
\end{proof}

In the next lemma, we see that we can always add $n$ to the beginning and 1 to the end of a shallow permutation of length $n-2$ and the result is a shallow permutation of length $n$, and we can similarly delete those elements from a shallow permutation of length $n$ to get a shallow permutation of length $n-2.$
\begin{lemma}\label{lem: n1}
    Let $\pi \in \S_{n-2}$ and $\tau = (n)(\pi_1+1)(\pi_2+1)\ldots(\pi_n+1)1$. Then $\pi\in \T_{n-2}$ if and only if $\tau\in\T_n$. 
\end{lemma}

\begin{proof}
    Let $\pi\in \T_{n-2}$. Then by Lemma \ref{lem: dir sum}, $\sigma=1\oplus\pi \in \T_{n-1}$. By Theorem \ref{thm: shallowRecursive} and with $\sigma_1=1$ a left to right max, we apply the inverse recursion from Theorem \ref{thm: shallowRecursive} which replaces the 1 with $n$ and moves 1 to the end. Thus we arrive at a shallow permutation in the form of $\tau$ as defined in the statement of the Lemma. 
\end{proof}

By repeatedly applying this lemma, we can obtain the following corollary, which we will use frequently in this paper.

\begin{corollary}\label{cor:decreasing}
    The decreasing permutation $\delta_n=n(n-1)\ldots21$ is shallow.
\end{corollary}


\begin{remark}
    It is possible to prove a stronger version of the above results. If $\pi \in \mathcal{T}_n$, $\tau \in \mathcal{T}_m$ and $\pi_i = i$, then the inflation of $i$ by $\tau$ remains shallow. Indeed, it is straightforward to check that the sum of the inversion number and reflection length of the inflated permutation still equals its total displacement. Lemma \ref{lem: dir sum} is the inflation of $12$ by two shallow permutations and Lemma \ref{lem: n1} is the inflation of $321$ at $2$ by a shallow permutation. We do not use this stronger result, so omit its full proof.
\end{remark}

We end this section by noting that shallow permutations necessarily avoid certain mesh patterns (see \cite{BC11}). We will denote by $3\boldsymbol{4}\boldsymbol{1}2$ the permutation pattern $3412$ where the ``4'' is equal to $n$ and the ``1'' is equal to 1. For example, $\pi = 642981537$ contains the subsequence $4913$ which is a $3\boldsymbol{4}\boldsymbol{1}2$. It also contains $6815$ which is a 3412 pattern, but not a $3\boldsymbol{4}\boldsymbol{1}2$ pattern since the ``4'' in this pattern is not equal to $n=9.$

We denote by $\underline{3}41\underline{2}$ the permutation pattern $3412$ where the ``3'' occurs in the first position and the ``2'' occurs in the last position. For example, the permutation $\pi = 672198435$ contains the subsequence $6835$ which is a $\underline{3}41\underline{2}$ pattern since 6, which is the ``3'' in this pattern, appears in the first position and 5, which is the ``2'' in this pattern, appears in the last position. 

\begin{theorem}\label{thm: 3n12}
    If $\pi\in\T_n$, then $\pi$ avoids the patterns $3\boldsymbol{4}\boldsymbol{1}2$ and $\underline{3}41\underline{2}$.
\end{theorem}
\begin{proof}
    Let us proceed by contradiction. Suppose $\pi$ contains a $3\boldsymbol{4}\boldsymbol{1}2$ pattern. We will show that upon repeatedly applying the right operator $R$, we will eventually move an element that will be neither a left-to-right maximum nor a right-to-left minimum in the new permutation, contradicting that $\pi$ is shallow. 

    To this end, suppose $\pi_r\pi_i\pi_j\pi_s$ is a $3\boldsymbol{4}\boldsymbol{1}2$ pattern, so $\pi_i=n$, $\pi_j=1$, and $\pi_r>\pi_s.$ Notice that when we apply the right operator once, we get
    \[\pi^R = \pi_1\ldots \pi_{i-1}\pi_n\pi_{i+1}\ldots \pi_{j-1}1\pi_{j+1}\ldots\pi_{n-1}.\]
    If $s=n$, then we have a contradiction since $1<\pi_s<\pi_r$ and so is neither a left-to-right maximum nor a right-to-left minimum in $\pi^R.$ If $s<n$, then we must have that $\pi_n$ is a left-to-right maximum, or else $\pi$ would not be shallow, so the element in position $i$ is still larger than all elements in positions 1 through $i-1$.
    
    Now, let us continue to apply the right operator, $R$. Each time, the last element is either deleted (if it is the largest element), moved to a position to the right of 1 (if the largest element is also to the right of 1), or it is moved to the left of 1, in which case it must be a left-to-right maximum. Note that each time an element is moved to the left of 1, it must be in a position greater than or equal to $i$ since each element moved over is itself larger than all elements in positions 1 through $i-1.$ Eventually, $\pi_{s}$ will be moved to the left of 1, and it will be moved to a position greater than or equal to $i$. However, $\pi_s<\pi_r$ with $r<i$. Thus $\pi_s$ cannot be a left-to-right maximum in this permutation. It also cannot be a right-to-left minimum since $1$ is to its right. Thus the original permutation is not shallow. 

    The other avoidance follow from Proposition~\ref{prop: invRc} and the fact that $\pi$ avoids $3\boldsymbol{4}\boldsymbol{1}2$ if and only if $\pi^{-1}$ avoids $\underline{3}41\underline{2}$.
\end{proof}

\section{Shallow permutations that avoid 132 or 213}\label{sec:132}

In this section, we enumerate shallow permutations that avoid the pattern 132. We also consider the number
of such permutations with a given number of descents, as well as those that exhibit certain symmetry. Let $\mathcal{T}_n(\sigma)$ denote the permutations $\pi \in S_n$ that are shallow and avoid $\sigma$. We set $t_n(\sigma) = |\mathcal{T}_n(\sigma)|$. Note that by Proposition \ref{prop: invRc} $\mathcal{T}_n(132) = \mathcal{T}_n(213)$, so proving Theorem~\ref{thm: 132} for shallow permutations avoiding 132 holds for 213 as well.

\subsection{Enumeration of 132-avoiding shallow permutations}
In this subsection, we will prove the following theorem. 
\begin{theorem} \label{thm: 132}
    For $n \geq 1$ and $\sigma \in \{132, 213\}$, $t_n(\sigma) = F_{2n-1}$, the $(2n-1)$st Fibonacci number.
\end{theorem}

We will first establish a few lemmas. This first lemma guarantees that for any shallow 132-avoiding permutation $\pi$, we must have that if $\pi$ does not start or end with $n$, it must end with $1$ and in the case that $n$ is not in the second-to-last position, $\pi$ must start with $(n-1).$

\begin{lemma} \label{lem: 132-ends-in-1}
    For $n\geq 3$, suppose $\pi \in \mathcal{T}_n(132)$.
    \begin{itemize}
        \item If $\pi_j =n$ with $2 \leq j \leq n-1$ then $\pi_n = 1$, and
        \item If $\pi_j = n$ with $2\leq j \leq n-2$ then $\pi_1 = n-1$.
    \end{itemize}
\end{lemma}

\begin{proof}
Let us consider the first bullet point.
Suppose $\pi \in \mathcal{T}_n(321)$ with $\pi_j=n$ for $2\leq j\leq n-1$. Note that $\pi_i>\pi_k$ for any $i<j<k$ since $\pi$ avoids 132.
    By Theorem \ref{thm: shallowRecursive}, $\pi^R_j=\pi_n$ must be either a left-to-right maximum or right-to-left minimum. It cannot be a left-to-right maximum because $\pi^R_{j-1} = \pi_{j-1} > \pi_n = \pi^R_j$. So $\pi^R_j$ must be a right-to-left minimum. However, since $\pi$ is $132$ avoiding we know that $1$ appears to the right of $n$ in $\pi$, so the only way for $\pi^R_j$ to be a right-to-left minimum is if $\pi^R_j = 1$, and thus $\pi_n = 1$.

    Now let us prove the second bullet point. Since $\pi$ is $132$ avoiding and $j > 1$, $n-1$ must occur to the left of $n$ in $\pi$. This means that $n-1$ occurs to the left of $1$ in $\pi^R$. Suppose $\pi^R_k = n-1$ with $1\leq k < j$, we will show that $k = 1$. Again by Theorem \ref{thm: shallowRecursive}, $\pi^{R^2}_k$ must be a left-to-right maximum or right-to-left minimum. But now it cannot possibly be a right-to-left minimum because $\pi^{R^2}_j = 1$ by Lemma \ref{lem: 132-ends-in-1} and $k < j \leq n-2$. So $\pi^{R^2}_k$ must be a right-to-left maximum. Since $\pi$ was $132$ avoiding every entry to the left of $1$ in $\pi^R$ will be larger than every entry to the right of $1$. So the only way $\pi^{R^2}_k$ is a right-to-left maximum is if $k = 1$.
\end{proof}

\begin{proof}[Proof of Theorem \ref{thm: 132}.]
    Let $a_n = |\mathcal{T}_n(132)|$ and $b_n$ be the number of $\pi \in \mathcal{T}_n(132)$ ending with $1$. Notice that $b_n$ is also the number of $\pi \in \mathcal{T}_n(132)$ beginning with $n$ since $\pi$ is a 132-avoiding shallow permutation if and only if $\pi^{-1}$ is. By Lemma \ref{lem: 132-ends-in-1}, we know that each $\pi \in \mathcal{T}_n(132)$ either begins with $n$, ends with $n$ or ends with $1$. There are clearly $a_{n-1}$ such permutations that end in $n$ (by removing that fixed point) and by Lemma~\ref{lem: n1}, there are $a_{n-2}$ such permutation that start with $n$ and end with 1.  Thus it follows that
    \[
    a_n = a_{n-1} + 2b_n - a_{n-2}.
    \]
    Next, let us find a recurrence for $b_n$; let $\pi\in\mathcal{T}_n(132)$ with $\pi_n=1$ and consider the position of $n$. If $\pi_{n-1} = n$, then $\pi^R \in \mathcal{T}_n(132)$ ending in $1$ and so there are $b_{n-1}$ such permutations. 
    If $\pi_j = n$ for $2 \leq j \leq n-2$, then $\pi_1 = n-1$ and so $\pi^{RL}$ is the 132-avoiding permutation obtained by deleting 1 from the end and $n-1$ from the front. Since in $\pi^R$, $\pi_n=1$ is clearly a right-to-left minimum and in $\pi^{RL}$, $\pi_1=(n-1)'=n-2$ will clearly be a right-to-left minimum, this permutation is also shallow. Since the resulting permutation is any shallow 132-avoiding permutation
    that does not end in $n-2$, there are $a_{n-2} - a_{n-3}$ such permutations. Finally, if $\pi_1 = n$, there are clearly $a_{n-2}$ such permutations by Lemma~\ref{lem: n1}. 
    
    Altogether, we find
    \[
    b_n = b_{n-1} + 2a_{n-2} - a_{n-3}.
    \]
    Substituting this back into our recursion for $a_n$,
    \begin{align*}
        a_n &= a_{n-1} + 2(b_{n-1} + 2a_{n-2} - a_{n-3}) - a_{n-2}\\
            &= a_{n-1} + (a_{n-2} +2b_{n-1} - a_{n-3}) + 2a_{n-2} - a_{n-3}\\
            &= 2a_{n-1} + 2a_{n-2} - a_{n-3}
    \end{align*}
    which is precisely the recursion satisfied by $F_{2n-1}$. 
\end{proof}

\subsection{132-avoiding shallow permutation by descent number}
In this subsection, we will refine the enumeration of 132-avoiding shallow permutations by their descent number.
We first present the following lemma. 

\begin{lemma} \label{lem: 132-descent-inverse}
    If $\pi \in \S_n(132)$ has $k$ descents, then $\pi^{-1} \in \S_n(132)$ also has $k$ descents.
\end{lemma}

\begin{proof}
    We will proceed by strong induction on $n$. The result is clear for $n \leq 3$. So assume $n \geq 4$ and $\pi \in \S_n(132)$. This means $\pi = (\tau \oplus 1) \ominus \sigma$ for some $\tau \in \S_k(132)$ and $\sigma \in S_{n-k-1}(132)$. In this case, $\pi^{-1} = \sigma^{-1} \ominus (\tau^{-1} \oplus 1)$. By induction $\sigma^{-1}$ and $\tau^{-1}$ have the same number of descents as $\sigma$ and $\tau$, we lose one descent in position $j$ of $\pi$, but gain an additional descent in position $n-j$ of $\pi^{-1}$ that does not come from any of the descents in $\sigma^{-1}$ or $\tau^{-1}$. We therefore preserve the number of descents, the result follows by induction.
\end{proof}

\begin{example}
    Consider $\pi = 534621 \in \S_6(132)$ with $3$ descents. $\pi = (312 \oplus 1) \ominus (21)$ and $\pi^{-1} = 652314$ which is $(21)^{-1} \ominus ((312)^{-1} \oplus 1)$. The number of descents in $(21)^{-1}$ and $(312)^{-1}$ are the same as in $21$ and $312$, we lose the descent in position $4$ of $\pi$, but gain a descent in position $6-4 = 2$ of $\pi^{-1}$ when we transition from $\sigma^{-1}$ to $\tau^{-1}$.
\end{example}

We now adapt the proof of Theorem \ref{thm: 132} to keep track of descents to arrive at the following result.

\begin{theorem}
    For $n\geq 2$, the number of shallow, 132-avoiding permutations with $k$ descents is equal to \[\displaystyle\binom{2n-2-k}{k}.\]
\end{theorem}

\begin{proof}
    Let $a_{n,k}$ be the number of shallow, 132-avoiding permutations of length $n$ with $k$ descents and let $b_{n,k}$ be the number of such permutations that end with $\pi_n=1$. Note that by Lemma~\ref{lem: 132-descent-inverse}, $b_{n,k}$ is also the number of shallow, 132-avoiding permutations with $k$ descents that starts with $\pi_1=n.$ 

    As in the proof of Theorem~\ref{thm: 132}, we know that for any permutation $\pi\in\mathcal{T}_n(132),$ by Lemma~\ref{lem: 132-ends-in-1}, $\pi$ has either $\pi_n = 1, \pi_n = n$ or $\pi_1 = n$.  It is clear there are $a_{n-1,k}$ shallow, 132-avoiding permutations with $k$ descents ending in $n$ since adding an $n$ to the end preserves shallowness and does not change the number of descents. For $n\geq 3,$ it also clear that there are $a_{n-2,k-2}$ permutations that both begin with $n$ and end with 1, which is seen by deleting both $n$ and $1$ to obtain a shallow permutation that still avoids 132 and has two fewer descents.
    This means
    \[
    a_{n,k} = a_{n-1,k} + 2b_{n,k} - a_{n-2,k-2}.
    \]

    Now, let us consider those permutations with $\pi_n=1$. As before, if $\pi_{n-1}=n$, then there are still $k$ descents in $\pi^R$, which still ends in 1, and so $b_{n-1,k}$ permutations. If $\pi_j=n$ for $2\leq j \leq n-2$, then $\pi_1=n-1$ by Lemma~\ref{lem: 132-ends-in-1}. If $j=2$, then $\pi^{RL}$ has one fewer descent and begins with its largest element $n-2.$ If $3\leq j\leq n-2$, then   $\pi^{RL}$ has two fewer descents, and it must not end or begin with its largest element $n-2.$ Thus in total, there are $b_{n-2,k-1} + a_{n-2,k-2} - b_{n-2,k-2}-a_{n-3,k-2}$ permutations. Finally, if $\pi_1=n$, as stated above there are $a_{n-2,k-2}$ such permutations. In total, we have 
    \[b_{n,k} = 2a_{n-2,k-2} + b_{n-1,k} + b_{n-2,k-1} - b_{n-2,k-2}-a_{n-3,k-2}.\]
    We claim that $b_{n,k} = a_{n-1,k-1}.$ We will prove this claim by strong induction on $n$. It is straightforward to check this claim for values of $n\leq 3$ so let us assume $n\geq 4$. Note that 
    \begin{align*}
        b_{n,k} &= 2a_{n-2,k-2} + b_{n-1,k} + b_{n-2,k-1} - b_{n-2,k-2}-a_{n-3,k-2} \\
         & = 2b_{n-1,k-1} + a_{n-2,k-1} + b_{n-2,k-1} - a_{n-3,k-3} - b_{n-2,k-1} \\
         & = a_{n-2,k-1} + 2b_{n-1,k-1} - a_{n-3,k-3} \\
         & = a_{n-1,k-1},
    \end{align*}
    where the last equality follow from the recurrence for $a_{n,k}$ above. 
    Notice that by taking $b_{n,k}=a_{n-1,k-1}$, we now obtain a recurrence for $a_{n,k}$ as follows: 
    \[
    a_{n,k}=a_{n-1,k} + 2a_{n-1,k-1} - a_{n-2,k-2},
    \]
    which together with the initial conditions is exactly the recurrence satisfied by $a_{n,k} = \binom{2n-2-k}{k}.$
\end{proof}

\subsection{132-avoiding shallow permutations with symmetry}

In this subsection, we consider 132-avoiding shallow permutations that are involutions (so that $\pi=\pi^{-1}$), that are centrosymmetric (so that $\pi=\pi^{rc}$), and that are persymmetric (so that $\pi=\pi^{rci}$).
\begin{theorem}
    For $n\geq 1$, the number of shallow, 132-avoiding involutions of length $n$ is $F_{n+1}$, where $F_{n+1}$ is the $(n+1)$-st Fibonacci number. 
\end{theorem}

\begin{proof}
    Let $i_n$ be the number of shallow, 132-avoiding permutations of length $n$ that are involutions. We will show that $i_n=i_{n-1}+i_{n-2}$ and with initial conditions $i_1=1$ and $i_2=2$, we have the Fibonacci sequence shifted by 1. 

    There are clearly $i_{n-1}$ shallow, 132-avoiding involutions of length $n$ with $\pi_n=n$ since adding the fixed point $n$ to the end of an involution in $\mathcal{T}_{n-1}(132)$ gives us a permutation that is still an involution, still avoids 132, and is still shallow by Theorem \ref{thm: shallowRecursive}.

    If $\pi\in\mathcal{T}_{n-1}(132)$ does not end in $n$, then by Theorem~\ref{lem: 132-ends-in-1}, $\pi_1=n$ or $\pi_n=1$. However, if $\pi$ is an involution, then one of these will imply the other. 
    Note that by Lemma \ref{lem: n1}, we can add an $n$ to the beginning and 1 to the end of an involution in $\mathcal{T}_{n-1}(132)$, and the resulting permutation is still shallow. Additionally the permutation still avoids 132 and is still an involution since we have only added the 2-cycle $(1,n)$. Thus there are $a_{n-2}$ involutions in $\mathcal{T}_n(132)$ beginning with $n$ and ending with 1. The recurrence, and thus the result, follows. 
\end{proof}

\begin{theorem}
    For $n\geq2$, the number of 132-avoiding shallow centrosymmetric permutations is $\lceil (n+1)/2\rceil$.
\end{theorem}

\begin{proof}
    Notice that if $\pi$ avoids 132 and $\pi=\pi^{rc}$, $\pi$ must also avoid $132^{rc}=213.$
    By Lemma~\ref{lem: 132-ends-in-1}, we know that either $\pi_n=n$, $\pi_n=1$, or $\pi_1=n$. However, if $\pi=\pi^{rc}$, then $\pi_1=n$ implies that $\pi_n=1.$ Therefore, either $\pi_n=n$ or $\pi_1=n$ and $\pi_n=1$. In the first case, since $\pi_n=n$ and $\pi$ avoids 213, $\pi$ is the increasing permutation. In the second case, by Lemma~\ref{lem: n1}, by deleting $n$ and 1, we obtain a shallow 132-avoiding centrosymmetric permutation of length $n-2.$ Letting $c_n$ be the number of centrosymmetric permutations in $\mathcal{T}_n(132)$, we thus have $c_n = c_{n-2}+1$, which together with the initial conditions that $c_1=1$ and $c_2=2$ implies the result. 
\end{proof}

\begin{theorem}
    Let $p_n(132)$ be the number of 132-avoiding shallow persymmetric permutations and let $P_{132}(x)$ be the generating function for $p_n(132)$. Then
    \[P_{132}(x) = \frac{1-x^2+2x^3}{(1-x)(1-2x^2-x^4)}.\]
\end{theorem}

\begin{proof}
    Let $n\geq 5$ and let $\pi\in\mathcal{T}_n(132)$ with $\pi=\pi^{rc}$. We use Lemma~\ref{lem: 132-ends-in-1} to determine a few possible cases. First, it $\pi_n=n$, since $\pi=\pi^{rc},$ we must have $\pi_1=1$, which implies that $\pi$ is the increasing permutation. If $\pi_{n-1}=n$, then by Lemma~\ref{lem: 132-ends-in-1}, we must have $\pi_n=1$. Since $\pi=\pi^{rc}$, then $\pi_1=2,$ which implies that $\pi=2345\ldots n1$ since $\pi$ is 132-avoiding. Note that this permutation is clearly shallow.  Next, consider a permutation where $\pi_j=n$ for some $2\leq j\leq n-2$. By Lemma~\ref{lem: 132-ends-in-1}, this permutation must end with 1 and start with $n-1$. But this implies that $\pi_2=n$ and so $\pi = (n-1)n\pi_3\ldots\pi_{n-1} 1$. Note that $\pi^{RL}$ can be obtained by  deleting $n$ and $1$ from $\pi.$ This permutation still is still shallow, still avoids 132, and still is persymmetric, and furthermore begins with $n$. If we let $q_n(132)$ be the number of persymmetric permutations in $\T_n(132)$ that begin with $n$, then we thus have \[ p_n(132) = 2+q_n(132)+q_{n-2}(132).\]
    Similarly considering those that end with $1$ (or equivalently start with $n$, since $\pi$ is persymmetric if and only if $\pi^{-1}$ is), we clearly have $p_{n-2}(132)$ permutations that start with $n$ and end with $1$ since removing these will leave a persymmetric shallow permutation avoiding 132. Considering the cases above that begin with $1$ listed above, we have \[q_n(132) = 1+p_{n-2}(132)+q_{n-1}(132).\] Letting $Q_{132}(x)$ be the generating function for $q_n(132)$, taking into account the initial conditions, we get 
    \[Q_{132}(x) = x+\frac{x^4}{1-x} + x^2P_{132}(x) + x^2Q_{132}(x)\] and 
    \[
    P_{132}(x) = 1+x^2+x^3 + \frac{2x^4}{1-x} + (1+x^2)Q_{132}(x).
    \]
    Solving for $Q_{132}(x)$, plugging the result into the equation for $P_{132}(x)$, solving for $P_{132}(x),$ and then simplifying gives the result in the statement of the theorem.
\end{proof}


\section{Shallow permutations that avoid 231 or 312}\label{sec:231}

In this section, we enumerate shallow permutations that avoid the pattern 231. We also consider the number of such permutations with a given number of descents, as well as those that exhibit certain symmetry. Note that by Proposition \ref{prop: invRc} $\mathcal{T}_n(231) = \mathcal{T}_n(312)$, and so Theorem~\ref{thm:231} holds for shallow permutations avoiding 312 as well.

\subsection{Enumeration of 231-avoiding shallow permutations}
 
\begin{theorem}\label{thm:231}
    Let $T_{231}(x) = \sum_{n\geq 0}t_n(231)x^n$ be the generating function for $t_n(231)$. Then,
    \[T_{231}(x) = \frac{1-3x+2x^2-x^3-x^4-x^5}{1-4x+4x^2-2x^3-x^4-x^5}.\]
\end{theorem}

We will prove this theorem via a series of lemmas. First, we prove that permutations of a particular form built from decreasing permutations are shallow.

\begin{lemma}\label{lem:shallow-n,n-1}
    If $\pi\in\S_n$ is of one of the permutations below:
    \begin{itemize}
        \item $21\ominus(\delta_j\oplus\delta_k)$, 
        \item $\delta_i\ominus(1\oplus\delta_k)$, or
        \item $\delta_i\ominus(\delta_j\oplus 1)$,
    \end{itemize} where $i,j,k\geq 0$, then $\pi$ is a shallow permutation.
\end{lemma}
\begin{proof}
    For the first bullet point, notice that $\pi^{LL}(21\ominus(\delta_j\oplus\delta_k)) = \delta_{j-2} \oplus (12 \oplus \delta_{k})$ which is a direct sum of shallow permutations and is therefore shallow. Furthermore, $n'=n-1$ is a left-to-right maximum in $\pi^{L}(21\ominus(\delta_j\oplus\delta_k))$ and $(n-1)''=n-3$ is a left-to-right maximum in $\pi^{LL}(21\ominus(\delta_j\oplus\delta_k))$. Therefore, Theorem \ref{thm: shallowRecursiveL} implies the original permutation is shallow.

    We  prove the second and third bullet points by induction on the length of the permutation. Let us first consider the second bullet point, when $\pi = \delta_i\ominus(1\oplus\delta_k)\in\S_n$. If  $k=0$, then $\pi = \delta_{i+1}$ which is shallow, and if $i=0$, then $\pi$ is a direct sum of shallow permutations and thus is shallow. Therefore, let us consider the cases when $i,k \geq 1$. It is straightforward to check the base cases when $n\leq 3,$ so let us assume $n\geq 4.$ Notice that $\pi^L= (\delta_{i-1}\oplus 1)\ominus \delta_k$ and $\pi^{LR} = (\delta_{i-1} \ominus (1 \oplus \delta_{k-1}))$. Since $n'=n-1$ is a left-to-right maximum of $\pi^L$, $2'=1$ is a right-to-left minimum of $\pi^{LR}$, and $\pi^{LR}$ is shallow by induction, we conclude by Theorems \ref{thm: shallowRecursive} and \ref{thm: shallowRecursiveL} that $\pi$ is also shallow. The result follows by induction.
 
    An identical argument works for the third bullet point since $(\delta_i \ominus(\delta_j \oplus 1))^{LR} = (\delta_{i-1}\ominus(\delta_{j-1} \oplus 1)$.
\end{proof}

In order to enumerate $\T_n(231)$, we will decompose these permutations into a direct sum of two shallow permutations that avoid 231, one of which begins with $\pi_1=n$. In order to enumerate those permutations in $\T_n(231)$ that begin with $n$ we will decompose them further, enumerating them in terms of those that begin with $\pi_1\pi_2=n(n-1)$. 

\begin{lemma}\label{lem:bc-231}
    Suppose $b_n$ is the number of permutations in $\T_n(231)$ with $\pi_1=n$ and let $c_n$ be the number of permutations in $\T_n(231)$ with $\pi_1=n$ and $\pi_2=n-1.$ Then we have \[b_n=\sum_{i=1}^{n-1} b_ic_{n-i+1}.\]
\end{lemma}

\begin{proof}
    We will show that if $\alpha\in\T_m(231)$ satisfies $\alpha_1=m$ and $\beta\in\T_\ell(231)$ with $\beta_1=\ell$ and $\beta_2=\ell-1$ then the permutation $\pi$ with $\pi_1=m+\ell-1$, $\pi_i=\alpha_i$ for $2\leq i\leq m$ and $\pi_{m+j-1}=\beta_{j}+m-1$ for $2\leq j\leq \ell$. In other words, taking $n=m+\ell-1$, we have \[\pi=n\alpha_2\alpha_3\ldots \alpha_m(n-1)\beta_3'\beta_4'\ldots \beta_\ell'\] where $\beta'_i=\beta_i+m-1$ for $3\leq i\leq \ell$. Let us first see that this permutation is also shallow. 

    Note that since $\alpha$ and $\beta$ are shallow, we have that $I(\alpha) + m-\cyc(\alpha) = D(\alpha)$ and $I(\beta) + \ell-\cyc(\beta) = D(\beta)$. It will be enough for us to show that $I(\pi) + n-\cyc(\pi) = D(\pi)$. 

    First, notice that $I(\pi) = I(\alpha)+I(\beta).$ Indeed, if $(i,j)$ is an inversion of $\pi$ (so that $i<j$ and $\pi_i>\pi_j$), then we have a few cases to consider. If $1\leq i,j\leq m$, then $(i,j)$ is also an inversion of $\alpha$, and in fact, all inversions of $\alpha$ are counted this way. If $(1,j)$ is an inversion of $\pi$ with $m+1\leq j \leq n$, then $(1,j-m+1)$ is an inversion of $\beta$ (since $\pi_1=n$). If $(i,j)$ is an inversion of $\pi$ with $m+1\leq i,j\leq n$, then $(i-m+1,j-m+1)$ is an inversion of $\beta$. Furthermore, the previous two cases count all inversions of $\beta$. Finally, since $\pi_r<\pi_s$ for all $2\leq r\leq m$ and $m+1\leq s\leq n$, there are no other inversions of $\pi.$

    Next, let us show that $\cyc(\pi) =\cyc(\alpha)+\cyc(\beta)-1.$ Notice that any cycles of $\alpha$ that do not contain $1$ are still cycles of $\pi$ since their values and positions are unchanged. Similarly, all cycles of $\beta$ that do not contain $1$ correspond to cycles of $\pi$ with values scaled up by $m-1.$ Let $(1, m, a_3, \ldots,a_r)$ and $(1,\ell, b_3,\ldots,b_s)$ be the cycles in $\alpha$ and $\beta$, respectively, that contain $1.$ Then in $\pi$, we have the corresponding cycle $(1,n, b_3+m-1, \ldots, b_s+m-1, m+1, a_3,\ldots,a_r)$.

    Finally, let us consider displacement; we will see that $D(\pi) = D(\alpha)+D(\beta).$ Indeed we have \begin{align*}
        D(\pi) &= \sum_{i=1}^n|\pi_i-i| \\
        &=(n-1) + \sum_{i=2}^m|\pi_i-i| + \sum_{i=m+1}^n|\pi_i-i| \\
        & =(n-1) + \sum_{i=2}^m|\alpha_i-i| + \sum_{j=2}^\ell|\beta_j+m-1-(m+j-1)| \\
        & =(n-1) + D(\alpha) -(m-1) + D(\beta)-(\ell-1) \\
        & = D(\alpha)+ D(\beta),
    \end{align*}
    where the last equality holds since $n=m+\ell-1.$

    Taken altogether, we can see that \begin{align*}
        I(\pi) +n-\cyc(\pi) &= I(\alpha)+I(\beta) + (m+\ell-1) -(\cyc(\alpha)+\cyc(\beta)-1) \\
         &= I(\alpha)+m-\cyc(\alpha)+I(\beta) + \ell -\cyc(\beta) \\
        &= D(\alpha)+D(\beta) \\
        &=D(\pi). 
    \end{align*}
\end{proof}

\begin{remark}
    One could also use Berman and Tenner's characterization of shallow permutations in \cite{BT22} to prove Lemma \ref{lem:bc-231} by considering the cycle form of $\pi$. We opted for a different proof to avoid introducing additional terminology.
\end{remark}

\begin{lemma}\label{lem:c-231}
    Let $c_n$ be the number of permutations in $\T_n(231)$ with $\pi_1=n$ and $\pi_2=n-1.$ Then for $n\geq 5$, $c_n=3n-11$.
\end{lemma}

\begin{proof} Let $n\geq 5$ and $\pi$ be a shallow permutation that avoids 231. 
Let us first consider permutations $\pi = n(n-1) \pi_3\ldots \pi_n$ so that $\pi_{k+3}=n-3$ for some $1\leq k \leq n-4$. Thus we have 
\[
\pi = n(n-1)\pi_3\ldots\pi_{k+2}(n-2)\pi_{k+4}\ldots \pi_n
\]
where $\{\pi_3,\ldots\pi_{k+2}\}=\{1,2,\ldots,k\}$ and $\{\pi_{k+4}, \ldots, \pi_n\}=\{k+1,\ldots,n-3\}$ since $\pi$ avoids 231. Furthermore, suppose $\pi_s=1$ for some $3\leq s\leq k+3$.

Notice that $\pi^L$ deletes $n$ from the beginning and replaces $\pi_s=1$ with the first element $n$ and re-sizes the elements, so that 
\[\pi^{L}=(n-2)(\pi_3-1)\ldots (\pi_{s-1}-1)(n-1)(\pi_{s+1}-1)\ldots (\pi_{k+2}-1)(n-3)(\pi_{k+4}-1)\ldots (\pi_n-1).  \] If the original permutation $\pi$ is shallow, then $\pi^L$ is as well since $n-1$ is necessarily a left-to-right maximum in a permutation in $\S_{n-1}$. Next, we find $\pi^{LR}=(\pi^L)^R$ by replacing $n$ in $\pi^L$ with the last element $(\pi_n-1)$ and deleting $(\pi_n-1)$ from the end. This cannot be a left-to-right maximum in $\pi^{LR}$ since $\pi^{LR}$ necessarily starts with its largest element. Notice it can only be a right-to-left minimum if $\pi_n$ is the smallest element among $\{\pi_{k+4}, \ldots, \pi_n\}$ and if the largest element in $\pi^L$ appeared after all elements smaller than the last element of $\pi^L$. In other words, $\pi_n=k+1$ and $\pi_{k+2}=1$. Since $\pi$ avoids 231, this implies that \[\pi = n(n-1)k(k-1)\ldots 21(n-2)(n-3)\ldots (k+1).\]
A similar argument proves that if $\pi$ ends with $n-2$, it must be of the form \[\pi = n(n-1)(n-3)(n-4)\ldots 21(n-2).\] 
Since by Lemma~\ref{lem:shallow-n,n-1}, these permutations are shallow,
this gives us $n-3$ shallow permutations $\pi$ that avoid $231$ and begin with $\pi_1\pi_2=n(n-1)$ with the property that $\pi_3\neq n-2$. Next we need to show that there are $2n-8$ such permutations with $\pi_3=n-2.$

Suppose that $\pi = n(n-1)(n-2)(n-3)\ldots (n-m) \pi_{m+2}\ldots \pi_{s-1}(n-m-1)\pi_{s+1}\ldots \pi_n$ for some $m\geq 3$ and $m+3\leq s\leq n$. We will show that there is one shallow permutation avoiding 231 with $s=m+3$ and one with $s=n.$
First suppose $\pi_n=n-m$. Then by the same argument above (i.e.~by considering the shallowness of $\pi^{LR}$), we must have that $\pi = n(n-1)\ldots (n-m)(n-m-2)\ldots 21(n-m).$
If $s=m+3$, then by the same argument as above, $\pi = n(n-1)\ldots (n-m)1(n-m-1)(n-m-2)\ldots 32.$
Note that by Lemma~\ref{lem:shallow-n,n-1}, these are both shallow permutations. 

Now for the sake of contradiction, suppose $m+3<s<n$. Then, by the argument above, $\pi = n(n-1)\ldots(n-m)k(k-1)\ldots 21(n-m-1)(n-m-2)\ldots (k+1)$ for some $k\geq 2.$ By considering $\pi^{LL}$, we get a permutation in $\S_{n-2}$ equal to 
\[\pi^{LL} = (n-2)'\ldots(n-m)'k'(k-1)'\ldots 3'(n-1)'n'(n-m-1)'(n-m-2)'\ldots (k+1)'\] where $j'=j-2$ for each element $3\leq j\leq n$.
Now taking $\pi^{LLR} = (\pi^{LL})^R$, we get
\[\pi^{LL} = (n-2)'\ldots(n-m)'k'(k-1)'\ldots 3'(n-1)'(k+1)'(n-m-1)'(n-m-2)'\ldots (k+2)'.\]
Finally, we consider $\pi^{LLRR}.$ First suppose $k<n-m-2$. Since $\pi^{LLRR}$ must start with its largest element $(n-2)'=n-4$, the element $(k+2)'=k$ must not be a left-to-right maximum. However, since it is to the left of $(k+1)'=k-1$ it is also not a right-to-left minimum and thus the permutation $\pi$ is not shallow. If $k=n-m-2$, then $\pi^{LLR}$ ends with $(n-m-1)'$, which is also smaller than $(n-2)'$ and larger than $(k+1)',$ and so will not be a left-to-right maximum or right-to-left minimum in $\pi^{LLRR}.$ Thus there are $2(n-4)$ shallow permutations avoiding 231 starting with $\pi_1\pi_2\pi_3=n(n-1)(n-2).$ 

Since we have a total of $n-3+2(n-4)$ shallow permutations that begin with $\pi_1\pi_2=n(n-1),$ the proof is complete. 
\end{proof}

We now have the tools necessary to prove the main theorem. 
\begin{proof}[Proof of Theorem~\ref{thm:231}]
    As above, suppose $b_n$ is the number of permutations in $\T_n(231)$ with $\pi_1=n$ and let $c_n$ be the number of permutations in $\T_n(231)$ with $\pi_1=n$ and $\pi_2=n-1.$ Let $B(x) = \sum_{n\geq1} b_nx^n$ and $C(x) = \sum_{n\geq 2} c_nx^n.$ 
    
    Since any $231$-avoiding permutation is the direct sum of a 231 avoiding permutation and a 231-avoiding permutation starting with $n$, we can use Lemma~\ref{lem: dir sum} to write that $t_n(231)=\sum_{i=0}^{n-1}t_i(231)b_{n-i}.$ Therefore, we have $T(x)=T(x)B(x)+1$.

    By Lemma~\ref{lem:bc-231}, we also have that $B(x) = \frac{1}{x}B(x)C(x)+x$.
    Finally, by Lemma~\ref{lem:c-231}, we know that for $n\geq 5,$ $c_n=3n-11$. Together with the fact that $c_2=1, c_3=1$, and $c_4=2,$ we have that 
    \[C(x) = x^4 + \frac{x^2}{1-x} + \frac{3x^5}{(1-x)^2}.\]
    Since $T(x) = \dfrac{1}{1-B(x)}$ and $B(x) =\dfrac{x}{1-\frac{1}{x}C(x)},$ the result follows.
\end{proof}

\subsection{231-avoiding shallow permutations by descent number}

We can refine the generating function in the previous section with respect to descents. Notice that since $312=231^{rc}$ and the reverse-complement preserves the number of descents, this result holds for 312-avoiding shallow permutations as well. 

For the purposes of this subsection, let $t_{n,k}(231)$ be the number of permutations in $\T_n(231)$ with $k$ descents, let $b_{n,k}$ be the number of such permutations that begin with $\pi_1=n$, and let $c_{n,k}$ be the number of such permutations that begin with $\pi_1\pi_2=n(n-1).$ Furthermore, let $T_{231}(x,t) = \sum t_{n,k}(231)x^nt^k$, $B(x,t) = \sum b_{n,k}x^nt^k$, and $C(x,t) = \sum c_{n,k}x^nt^k$.
\begin{theorem}
    \[C(x,t) = t^2x^4 + \frac{tx^2}{1-xt} + \frac{3t^3x^5}{(1-xt)^2}\]
     and 
     \[B(x,t) = \frac{x+C(x,t)-\frac{1}{t}C(x,t)}{1-\frac{1}{xt}C(x,t)} \]
     and finally, \[T(x,t) = \frac{1}{1-B(x,t)}.\]
\end{theorem}
\begin{proof}
    We first note that by the proof of Lemma~\ref{lem:c-231}, shallow permutations that avoid 231 and begin with $\pi_1\pi_2=n(n-1)$ must either be the decreasing permutation or have at most one ascent. It follows that for each $n$, the coefficient of $x^n$ in $C(x,t)$ must be the polynomial $(3n-10)t^{n-2} + t^{n-1}$ for $n\geq 5.$ It follows that
   \[C(x,t) = t^2x^4 + \frac{tx^2}{1-xt} + \frac{3t^3x^5}{(1-xt)^2}.\]
   Next, by the proof of Lemma~\ref{lem:bc-231}, permutations in $\T_n(231)$ that start with $n$ are built from smaller permutations: $\alpha$ that starts with $n$ and $\beta$ that starts with $n(n-1).$ When the $\alpha$ is at least size 2, we have that $\des(\pi) = \des(\alpha)+\des(\beta) -1$ since the first descent in $\beta$ is lost in this process. Therefore, we get that 
     \[B(x,t) = x+C(x,t) +\frac{1}{xt}C(x,t)(B(x,t)-x). \]
     Finally, the number of descents in the direct sum of two permutations is the sum of the number of descents in each summand. Therefore $T(x,t) = T(x,t)B(x,t)+1$.
\end{proof}

\subsection{231-avoiding shallow permutations with symmetry}

In this subsection, we consider those 231-avoiding shallow permutations that exhibit certain symmetries. In particular, we enumerate 231-avoiding shallow involutions, in which $\pi=\pi^{-1}$, 231-avoiding shallow centrosymmetric permutations, in which $\pi=\pi^{rc},$ and 231-avoiding shallow persymmetric permutations, in which $\pi=\pi^{rci}.$ 
We show that in fact all 231-avoiding involutions and centrosymmetric permutations are shallow, but this same result does not hold for persymmetric permutations. 

\begin{theorem}
    For $n\geq 1$, the number of shallow, 231-avoiding involutions of length $n$ is $2^{n-1}$. 
\end{theorem}

\begin{proof}
    In \cite{SS85}, Simion and Schmidt show there are $2^{n-1}$ involutions of length $n$ that avoid 231. In their proof, it is shown that each of these permutations is a direct sum of decreasing permutations, i.e., $\pi = \bigoplus_{i=1}^k \delta_{m_i}$ for some composition $(m_i)_{i=1}^k$ of $n$. Since the decreasing permutation is always shallow, as is the direct sum of shallow permutations by Lemma \ref{lem: dir sum}, all 231-avoiding involutions are shallow.
\end{proof}

\begin{theorem}
    For $n\geq 1$, the number of shallow, 231-avoiding centrosymmetric permutations of length $n$ is $2^{\lfloor n/2\rfloor}$. 
\end{theorem}

\begin{proof}
    In \cite{E07}, Egge shows there are $2^{\lfloor n/2\rfloor}$ centrosymmetric permutations of length $n$ that avoid 231. In his proof, it is shown that each of these permutations is a direct sum of decreasing permutations, i.e., $\pi = \bigoplus_{i=1}^k \delta_{m_i}$ for a palindromic composition $(m_i)_{i=1}^k$ of $n$. Since the decreasing permutation is always shallow, as is the direct sum of shallow permutations by Lemma \ref{lem: dir sum}, all 231-avoiding centrosymmetric permutations are shallow.
\end{proof}

\begin{theorem}
    For $n\geq 1$, if the number of shallow, 231-avoiding persymmetric permutations of length $n$ is $p_n(231)$ and the corresponding generating function is $P_{231}(x)$, then 
    \[P_{231}(x) =  \frac{x^{10} + 2 x^8 + x^7 + x^6 - x^5 - 2 x^4 + x^3 + 2 x^2 - x - 1}{x^{10} + x^8 + 2 x^6 - 4 x^4 + 4 x^2 - 1}.\]
\end{theorem}

\begin{proof}
    Let $P^B(x)$ be the generating function for shallow 231-avoiding persymmetric permutations that begin with $n$ and $P^C(x)$ be the generating function for those beginning with $n(n-1)$. Then, since the only 231-avoiding shallow permutations that begin with $n(n-1)$ (of the form described in Lemma~\ref{lem:c-231}) are the decreasing permutation $\pi = n(n-1)\ldots 21,$ the permutations $\pi = n(n-1)\ldots 4312$, and when $n$ is even, the permutation $21\ominus(\delta_{n/2-1}\oplus\delta_{n/2-1}).$ Therefore for $n\geq 6,$ there are 2 such permutations when $n$ is odd and 3 such permutations when $n$ is even, giving us \[P^C(x) = x^2+x^3+\frac{2x^4}{1-x} + \frac{x^6}{1-x^2}.\]
    For those permutations beginning with $n$, if $\pi_i=n-1$ for $i>2$, then we must have that $\pi_i\pi_{i+1}\ldots \pi_n$ are composed of the numbers $\{i-1,i,\ldots,n-1\}$ and is order-isomorphic to the reverse-complement-inverse of $\pi_2\pi_3\ldots \pi_{n-i+2}$ which is composed of the elements in $\{1,2,\ldots, m-1\}.$ The remaining permutation is itself a shallow 231-avoiding persymmetric permutation beginning with $n.$ Thus, we have that \[P^B(x) = x+ P^C(x) +\frac{1}{x^2}C(x^2)P^B(x)\] where $C(x)$ is the generating function given in the proof of Theorem~\ref{thm:231}. Finally, if a given persymmetric permutation in $\T_n(231)$ does not begin with $n$, it is the direct sum $\gamma\oplus\nu\oplus\gamma^{rci}$ where $\nu$ is a shallow 231-avoiding persymmetric permutation and $\gamma$ is any shallow 231-avoiding permutation beginning with $n$. Thus, \[P_{231}(x) = 1+P^B(x) + B(x^2)T'(x)\] where $B(x)$ is the generating function given in the proof of Theorem~\ref{thm:231}. The result follows. 
\end{proof}

\section{Shallow permutations that avoid 123}\label{sec:123}

In this section, we consider those shallow permutations that avoid the pattern 123, as well as those that exhibit the three symmetries of inverse, reverse-complement, and reverse-complement-inverse. We omit the enumeration of 123-avoiding shallow permutations with a given number of descents, though this is likely tractable (but tedious) by following the proof of Theorem~\ref{thm: 123} below. 

\subsection{Enumeration of 123-avoiding shallow permutations}



Let us start by stating the main theorem in this section. 
\begin{theorem}\label{thm: 123}
Let $T_{123}(x)$ be the generating function for the number of shallow permutations that avoid 123. Then, 
    \[
T_{123}(x) =\frac{1-3x+11x^3-13x^4+7x^5+6x^6+3x^7}{(1-x)^4 (1 - 4 x^2 + x^4)}.
\]
\end{theorem}

We first establish a few lemmas based on the position of $n$ and $1$ in the permutation. In Lemma~\ref{lem: 123-1}, we consider those permutations that do not start with $n$ or end with 1, and in Lemma~\ref{lem: 123-2}, we consider those that do start with $n$ and have a 1 in any other position.

\begin{lemma}\label{lem: 123-1}
    For $n\geq 3$, the number of 123-avoiding shallow permutations with $\pi_1\neq n$ and $\pi_n\neq 1$ is equal to $\displaystyle2\binom{n-1}{3} + (n-1).$
\end{lemma}
\begin{proof}
    Let us first consider the case when $\pi_i=n$ and $\pi_j=1$ for some $1<i<j<n.$ We will see that there are $j-i$ such 123-avoiding shallow permutations. In particular, these $j-i$ permutations are of the form 
    \[\pi=\underline{(t-1)\ldots(t-i+1)}\, n\, \underline{(t-i)\ldots 2} \,\underline{(n-1) \ldots (t+n-j)}\, 1\, \underline{(t+n-j-1)\ldots t}\] for any $i+1\leq t\leq j$ where the underlined regions are decreasing. 

    We will first show that $\pi$ is shallow. Let us consider the permutation $\pi^{R^{n-t}}$. Since upon each iteration of the right operator, the last element replaces the largest element, all elements that appear before $n-1$, except for $n$, will remain unchanged. Each time, a term will be deleted, leaving us with \[\pi^{R^{n-t}} = \underline{(t-1)\cdots(t-i+1)}t\underline{(t-i)\cdots 21} \in \S_t.\] For example, if $\pi = 493287165$, we have $n = 9$ and $t = 5$, so $\pi^{R^4} = 45321$. In the first step, $t$ is a left-to-right maximum in $\pi^R$, and in all the subsequent steps the element we move is a right-to-left minimum in its new position. Furthermore, $\pi^{R^{n-t}} = (\delta_{i-1} \oplus 1) \ominus \delta_{t-i}$ is shallow by an identical argument to Lemma \ref{lem:shallow-n,n-1}. These two facts in combination with Theorem \ref{thm: shallowRecursive} imply that $\pi$ is shallow.

    Now let us see that these are indeed the only shallow 123-avoiding permutations with $\pi_i=n$ and $\pi_j=1$ for some $1<i<j<n.$ Indeed, since $\pi$ avoids 123, we must have $\pi_1\ldots \pi_{i-1}$ and $\pi_{j+1}\ldots \pi_n$ are decreasing. Furthermore, by considering $\pi^R,$ we would have that $\pi_n$ is to the left of 1 and thus must be a left-to-right maximum, implying that $\pi_n>\pi_1,$ which in turn implies that each element of $\{\pi_1,\ldots,\pi_{i-1}\}$ is less than each element of $\{\pi_{j+1},\ldots,\pi_n\}$. This implies that if $\pi_r=2$ then either $r=i-1$ or $i<r<j$. Clearly if $\pi_{i-1}=2$, then the subsequence $\pi_{i+1}\ldots\pi_{j-1}\pi_{j+1}\ldots \pi_n$ is decreasing and thus is of the above form with $t = i+1$. Similarly, if $\pi_s = n-1$, then either $s=j+1$ or $i<s<j$. If $\pi_{j+1}=n-1,$ then $\pi$ must be of the form above with $t=j.$ We can thus assume $i<r,s<j$. If $r<s$, then it is of the form above, so for the sake of contradiction, suppose $r>s$ (so, suppose 2 appears after $n-1$). However, in this case, $\pi^{RL}$ contains the subsequence $\pi_n'(n-1)'2'\pi_1'$ which is a $3\boldsymbol{4}\boldsymbol{1}2$ pattern, contradicting Theorem~\ref{thm: 3n12}. 

    Next, let us consider those permutations with $1$ appearing before $n$ in $\pi.$ Since $\pi$ avoids 123, it must be that $\pi=\pi_1\ldots \pi_{i-1}1n\pi_{i+2}\ldots \pi_n$ for $1\leq i \leq n-1$. Furthermore, we must have $\pi_1 > \pi_2 > \cdots > \pi_{i-1}$ and $\pi_{i+2} > \pi_{i+3} > \cdots > \pi_n$.

    We claim that if $\pi_i = 1$ and $\pi_{i+1} = n$, then either $\pi_1 < \pi_n$ in which case
    \[
    \pi = i(i-1) \cdots 1n(n-1)(n-2) \cdots (i+1),
    \]
    or $\pi_1 > \pi_n$. Since, the elements preceding $1$ are decreasing and those after $n$ are decreasing, we must have that $\pi_1 \in [i+1,n-1]$, $\pi_n \in [2,i]$. Furthermore, we can show that $\pi_{n-1} > \pi_2$. For the sake of contradiction, suppose not. Then $\pi_1 > \pi_2 > \pi_{n-1} > \pi_n$. But then $\pi^{RL}$ contains the sequence $\pi_2'\pi_1'\pi_n'\pi_{n-1}'$ which is a $\underline{3}41\underline{2}$ pattern, contradicting Theorem~\ref{thm: 3n12}.
    Thus once $i,$ $\pi_1$, and $\pi_n$ are selected, the rest of the permutation is determined.


    
    
    So in total for each $1 \leq i \leq n-1$ there are $(n-i-1)(i-1)$ permutations with $\pi_1 > \pi_n$ and 1 with $\pi_1<\pi_n.$ Summing over all possible values of $i$, we obtain
    \[
    \sum_{i=1}^{n-1} (1+(n-i-1)(i-1)) = \binom{n-1}{3} + (n-1)
    \]
    total permutations with $1$ appearing before $n$. 

    Altogether, there are $\sum_{j=3}^{n-1} \sum_{i=2}^{j-1} (j-i) = \binom{n-1}{3}$ permutations with $n$ appearing before $1$ and $\binom{n-1}{3} + (n-1)$ permutations where $1$ appears before $n$. Adding these gives us the result.
\end{proof}

Let $b_n$ be the number of permutations in $\T_n(123)$ that start with $\pi_1=n$ and let $b_n(j)$ be the number of such permutations that also have $\pi_j=1.$ Note that by considering the reverse-complement, we have that $b_n$ is also the number that end with $\pi_n=1$ and $b_n(j)$
 is also the number with $\pi_{n-j+1}= n$ and $\pi_n=1$. \begin{lemma}\label{lem: 123-2}
For $n\geq 5,$ we have $b_n(2) = 1$, $b_{n}(n-1) = b_n(n) = t_{n-2}(123)$, and for $3\leq  j\leq n-2$ we have 
\[
b_n(j) = 2 + (n-j-2)(2j-5) + 4\binom{j-2}{2} + b_{n-2}(j-1).
\]
\end{lemma}

\begin{proof}
    Let us first consider those permutations with $\pi_1=n$, $\pi_j=1$ and $\pi_n=2$ with $j\leq n-2.$ Notice that $\pi^{RL}$ is still shallow of length $n-2$ and has the property that 1 appears in the $j-1$ position where $j-1\leq n-3,$ so $\pi^{RL}$ does not end with 1. It avoids 123 since it was essentially obtained by ``deleting'' 1 and $n$. By considering the position of $n-2$ in $\pi^{RL}\in\T_{n-2}(123)$, by the proof of Lemma~\ref{lem: 123-1}, there are $1+\binom{j-2}{2} + (j-2)(n-2-j) + b_{n-2}(j-1)$ such permutations. 

    Next, let us consider those with $\pi_i=2$ with $1<i<j.$ First, let us consider those permutations with $\pi_{j+1}=n-1$. In this case, we must have $i=j-1$, so we have
    \[\pi=n\pi_2\ldots\pi_{j-2}21(n-1)\pi_{j+2}\ldots \pi_n\] where $\pi_2\ldots\pi_{j-2}$ and $\pi_{j+2}\ldots \pi_n$ are both decreasing segments since $\pi$ is 123-avoiding.  
    We claim that the only such permutations are either \[n(j-1)\ldots 21(n-1)\ldots j\]
    or that have $\pi_2\in\{j,\ldots,n-2\}$ and $\pi_n\in\{3,\ldots,j-1\}$, with all the remaining elements before 2 being smaller than all the remaining elements after 1. If $\pi$ is not of one of these forms, then we have $\pi_2>\pi_3>\pi_{n-1}>\pi_n,$ in which case $\pi^{LLR}$ would contain a $\underline{3}41\underline{2}$ pattern, contradicting Theorem~\ref{thm: 3n12}. These are clearly shallow since $\pi^{LLR}$ is the direct sum of two shallow permutations, and it
    is clear there are $(j-3)(n-j-1)+1$ such permutations based on the choices of $\pi_2$ and $\pi_{n}.$

    Next, consider those with $\pi_2=n-1,$ so 
    \[\pi=n(n-1)\pi_3\ldots\pi_{i-1}2\pi_{i+1}\ldots\pi_{j-1}1\pi_{j+1}\ldots \pi_n\]
    where $\pi_{i+1}\ldots\pi_{j-1}\pi_{j+1}\ldots \pi_n$ is decreasing since $\pi$ avoids 123. Notice that by first considering $\pi^{RR},$ we get a permutation \[\pi^{RR}=\pi_n\pi_{n-1}\pi_3\ldots\pi_{i-1}2\pi_{i+1}\ldots\pi_{j-1}1\pi_{j+1}\ldots \pi_{n-2}\]
    with $\pi_{n-1}>\pi_n$ since $j\leq n-2.$ This is clearly still shallow if the original $\pi$ was. Now, taking $\pi^{RRLL}$, we see that our original permutation is only shallow if $\pi_{n-1}$ is a left-to-right maximum in $\pi^{RRLL}$ since $\pi_n<\pi_{n-1}$ will appear to its right. Thus we must have that the elements of the segment $\pi_3\ldots\pi_{i-1}$ are all less than $\pi_{n-1},$ as is $\pi_{n}.$ Thus $\pi_{n-1}=i+1$ and there are $i-2$ choices of $\pi,$ all of which are shallow. Summing over all possible choices of $i,$ we see there are $\binom{j-2}{2}$ permutations. 

    Now left us consider the final case, when $\pi_j=1$, $\pi_i=n-1$ with $3\leq i\leq j-1$, and $\pi_n\neq 2.$ We claim that $\pi_n\in\{3,\ldots,j-1\}$ for each possible value of $i$ and that the other terms are determined, for a total of $(j-3)^2$ permutations.

    Indeed, in this case, we have $\pi=n\pi_2\ldots(n-1)\ldots 1\ldots\pi_{n-1}\pi_n$, and so $\pi^{RL} = \pi'_2\ldots(n-1)'\ldots \pi'_{n}\ldots\pi'_{n-1}$. Note that if we show that both $\pi_2<\pi_{n-1}$ and $2$ appears before $n-2$, the rest of the permutation $\pi$ must be determined since $\pi$ must avoid 123. Notice that in $\pi^{RL}$, if $\pi_2>\pi_{n-1},$ then $\pi'_2(n-1)'\pi_n'\pi_{n-1}'$ is a $\underline{3}41\underline{2}$ pattern, contradicting Theorem~\ref{thm: 3n12}. Note also that $\pi_2\neq n-2$ since otherwise $\pi_2'(n-1)'\pi_n'\pi_{n-1}'$ would be a $\underline{3}41\underline{2}$ pattern in $\pi^{RL}.$ If $n-2$ does occur before 2, then we would have \[\pi = n\pi_2\ldots (n-1)\ldots (n-2)\ldots 2\ldots 1\ldots\pi_{n-1}\pi_n,\] but then $\pi^{RLR}$ contains $\pi_{n-1}'(n-2)'2'\pi_{n}'$ which is a $3\boldsymbol{4}\boldsymbol{1}2$ pattern, contradicting Theorem~\ref{thm: 3n12}.

   Thus we have $(j-3)(n-j-1)+1 + \binom{j-2}{2} + (j-3)^2$ permutations that do not end in 2. 
   Adding all these possible cases togther gives us the result in the statement of the theorem.
\end{proof}

We are now ready to prove the main theorem of this section.

\begin{proof}[Proof of Theorem~\ref{thm: 123}]
Let $t_n(123)$ be the number of permutations in $\T_n(123)$ and let $b_n$ be the number of permutations in $\T_{n}(123)$ that start with $n$. Since there are clearly also $b_n$ permutations that end with 1 and $a_{n-2}$ permutations that both start with $n$ and end with $1$, using the results of Lemm~\ref{lem: 123-1}, we have \[t_n(123) = 2b_n-t_{n-2}(123) + 2\binom{n-1}{3}+n-1. \]
 Using Lemma~\ref{lem: 123-2}, we obtain
 \begin{align*}
 b_n &= \sum_{j=2}^n b_n(j) \\
 &= 1+2t_{n-2}(123)+\sum_{j=3}^{n-2}\left(2+(n-j-2)(2j-5) + 4\binom{j-2}{2} + b_{n-2}(j-1)\right) \\
 & =1+2t_{n-2}(123)+ 2(n-3) + 5\binom{n-3}{3}+\binom{n-4}{3} + b_{n-2}-b_{n-2}(n-2)\\
  & =1+2t_{n-2}(123)+ 2(n-3) + 5\binom{n-3}{3}+\binom{n-4}{3} + b_{n-2}-t_{n-4}(123).\\
 \end{align*}
 Thus if $B(x)$ is the generating function for the sequence $\{b_n\}$, we have
 \[
 T_{123}(x) = 2B(x) -x^2T_{123}(x) + \frac{x^2}{(1-x)^2} + \frac{2x^4}{(1-x)^4} + 1-x
 \]
 and
 \[
 B(x) = (2x^2-x^4)T_{123}(x) + x^2B(x) + \frac{x}{1-x} + \frac{5x^6+x^7}{(1-x)^4} + \frac{2x^5}{(1-x)^2}-2(x^2+x^3).
 \]
 Solving for $T_{123}(x)$, we obtain the result in the statement of the theorem.
\end{proof}






\subsection{123-avoiding shallow permutations with symmetry}

In this subsection, we consider 123-avoiding permutations that exhibit one of the three symmetries. 

\begin{theorem}
    For $n\geq 1$, the number of shallow, 123-avoiding involutions of length $n$ is $ \lfloor \frac{n^2}{4} \rfloor +1$.
\end{theorem}

\begin{proof}
    Let $a_n$ be the number of shallow, 123-avoiding permutations that are involutions. We will show that  $a_n=a_{n-2}+n-1$. This together with the initial conditions $a_1=1$ and $a_2=2$ implies the formula as given in the statement of the theorem. 
    
    Note that by Lemma \ref{lem: n1}, there are $a_{n-2}$ shallow 123-avoiding involutions that start with $n$ and end with $1$ since these comprise a 2-cycle and thus removing them leaves us with an involution. Also note that all involutions that have $\pi_n=1$ must also have $\pi_1=n$ and thus all involutions starting with $\pi_1=n$ are counted in this way.

    Next suppose $\pi_i=n$ for $i>1$. Then since $\pi$ is an involution $\pi_n=i.$ We claim that $\pi_1\leq \pi_n.$ For the sake of contradiction, suppose not. If $\pi_1>\pi_n=i$, then since $\pi$ is an involution $\pi_{\pi_1}=1$. Since $\pi_1>i$, this 1 appears after $n$ and before $\pi_n.$ Thus, in $\pi^L$, $\pi_1$ is replaces this 1, but cannot be a left-to-right maximum since $n$ is to its left and cannot be a right-to-left minimum since it is larger than $\pi_n.$ Thus $\pi_1\leq \pi_n.$

    Finally, since $\pi$ avoids 123 and $\pi_1\leq \pi_n,$ the only permutations that satisfy this are of the form \[\pi=m(m-1)\ldots 1 n (n-1)\ldots (m+1)\] for $m\in[n-1]$. There are clearly $n-1$ such permutations, adn so the result follows.
    
    
\end{proof}

\begin{theorem}
    For $n\geq 1$, the number of shallow, 123-avoiding centrosymmetric permutations of length $n$ is $  \frac{n^2}{4}  +1$ when $n$ is even and 1 when $n$ is odd.
\end{theorem}

\begin{proof}
Let $c_n$ be the number of centrosymmetric 123-avoiding permutations.
First, let us consider the case when $n$ is odd. Since $\pi=\pi^{rc},$ we must have that $\pi_{(n+1)/2}=(n+1)/2.$  Since $\pi$ avoids 123, it must be the case that the elements in $\pi_1\pi_2\ldots \pi_{(n+1)/2-1}$ are greater than $(n+1)/2$ and the elements in $\pi_{(n+1)/2+1}\ldots\pi_n$ are less than $(n+1)/2.$ In particular, $n$ occurs in the first half and 1 occurs in the second half. If 1 occurs at the end of $\pi,$ then since $\pi=\pi^{rc}$, $\pi_1=n.$ Thus by Lemma~\ref{lem: n1}, there are $c_{n-2}$ such permutations. If $1$ does not occur at the end, then $n$ necessarily does not occur at the beginning. But then, in $\pi^R,$ $\pi_n$ is neither a left-to-right maximum nor a right-to left minimum. Thus, when $n$ is odd, we have $c_n=c_{n-2}.$ Since $c_1=1$, the result for odd $n$ follows.

Now, suppose $n$ is even. We will show that $\pi$ either starts with $n$, in which case there are $c_{n-2}$ for the same reasons as above, or is either of the form \[\pi = (n-k)(n/2)(n/2-1)\ldots (k+1)(k-1)\ldots 21 n (n-1)\ldots (n-k+1)(n-k-1)\ldots (n/2+1) k \]
for $2\leq k\leq n/2+1$,
or of the form
\[\pi=(n/2)(n/2-1)\ldots (k+1) n k\ldots 2(n-1)(n-2)\ldots (n-k+1) 1(n-k)\ldots (n/2+1) \]
for $1\leq k< n/2$.

Let us first consider the case when $n$ appears after 1 in $\pi.$ Since $\pi$ avoids 123 and is centrosymmetric, it must be that $\pi_{n/2}=1$ and $\pi_{n/2+1}=n$. Note that if $\pi_1<\pi_n$, then we must have the first case above with $k=n/2+1$, so let us assume $\pi_1>\pi_n.$ In that case, in $\pi^{RL}$, we will have a $\underline{3}41\underline{2}$ pattern unless $\pi_2<\pi_{n-1}$, contradicting Theorem~\ref{thm: 3n12}. Since $\pi$ is centrosymmetric, the only possibility is the first one listed above. 

Next consider when $n$ appears before 1 in $\pi.$ In this case, we must have $\pi_1>\pi_n$ or else we will have a $3\boldsymbol{4}\boldsymbol{1}2$ pattern, contradicting Theorem~\ref{thm: 3n12}. Therefore, since $\pi$ avoids 123 and is centrosymmetric, we must have $\pi_1=n/2$ and $\pi_n=n/2+1$. Furthermore, the elements that appear before $n$ are decreasing and consecutive and those after 1 are decreasing are consecutive, since otherwise we would have a 123 pattern. This implies that either $1$ appears immediately after $n,$ in which case we have the second case above with $k=1,$ or the 2 and $n-1$ appear between the $n$ and $1$ in $\pi.$ In fact, we must have 2 appearing before $n-1$, or else $\pi^{RL}$ will have a $3\boldsymbol{4}\boldsymbol{1}2$ pattern, contradicting Theorem~\ref{thm: 3n12}. It is a straightforward exercise to check that these permutations listed above are indeed shallow, and we now have shown they are the only possible shallow 123-avoiding permutations of length $n.$

Thus when $n$ is even, $c_n=c_{n-2}+n-1,$ which together with the fact that $c_2=2,$ implies that $c_n = \frac{n^2}{4}+1.$
\end{proof}

\begin{theorem}
    For $n\geq 3$, the number of shallow 123-avoiding persymmetric permutations of length $n$ has the associated generating function 
    \[
    P_{123}(x) = \frac{x^6 + x^5 + x^3 - 2 x^2 + 1}{(x - 1)^2 (x + 1) (1-2x^2-x^4)}.
    \]
\end{theorem}
\begin{proof}
    Let $p_n$ denote the number of persymmetric 123-avoiding permutations and let $q_n$ denote those that start with $\pi_1=n.$

    First note that if $\pi_2=n$, then we must have $\pi_1=n-1$ since $\pi$ is persymmetric. Also, we must have $\pi_n=1$ since if 1 appeared anywhere else in $\pi,$ then in $\pi^L$, the element $n-1$ would not be a left-to right maximum nor a right-to-left minimum, and so $\pi$ would not be shallow. Thus since $\pi_1\pi_2=(n-1)n$ and $\pi_n=1,$ then $\pi^{RL}\in\S_{n-2}$ will be a shallow 123-avoiding persymmetric permutation that starts with $n-2.$ Since any such permutation can be obtained this way, there are $q_{n-2}$ persymmetric permutations $\pi\in\T_{n}(123)$ with $\pi_2=n.$

    Now, we will show there is exactly one shallow persymmetric 123-avoiding permutation with $\pi_i=n$ for $3\leq i \leq \lfloor \frac{n}{2}\rfloor +1$ and none with $i>\lfloor\frac{n}{2}\rfloor+1.$
    First note that if $i>\lfloor\frac{n}{2}\rfloor+1$, then $\pi_1 = n+1-i$. But since $\pi$ avoids 123, the elements before $n$ must be decreasing, which is impossible in this case since $\pi_1$ is too small. 
    Now assume $i\leq \lfloor\frac{n}{2}\rfloor+1.$ Since $\pi$ is persymmetric, this means $\pi_1=n+1-i$ and since $\pi$ avoids 123, we have $\pi_1\ldots\pi_{i-1}$ is decreasing. If the 1 appears before $n$, then we must have that $\pi_{i-1}=1$ and $\pi_n=n+2-i$, and that every element after $n$  is decreasing in order to avoid 123. The only way this is possible is if $i=\lfloor\frac{n}{2}\rfloor+ 1$ and $\pi = (n/2)\ldots 21n(n-1)\ldots (n/2+1)$. In fact, this is the only possibility for $i =\lfloor\frac{n}{2}\rfloor+ 1$, so assume $i\leq \lfloor\frac{n}{2}\rfloor$ and that the 1 appears after $n.$ Note that if $\pi_j=i$ for $i+1\leq j\leq n-1$, then $\pi_n=n+1-j$ which implies $\pi$ contains $3\boldsymbol{4}\boldsymbol{1}2$. In order to avoid this, we must have $\pi_n=1.$ Since $\pi^{RL}$ must avoid $\underline{3}41\underline{2}$, we must have that either $\pi_{n-1}=2$ or $\pi_{n-1}>\pi_2$. In the second case, since $\pi$ avoids 123 and persymmetric, the only possibility is that $n$ is odd and taking $r=(n+1)/2$ we get \[\pi = (n+1-j)(r-1)(r-2)\ldots (r-i+2)n(r-i+1)\ldots 32(n-1)(n-2)\ldots r1.\]
    If $\pi_{n-1}=2$, then we must not have $\pi_{n-2}=3$ since $\pi^{RRLL}$ would send $\pi_2$ to where $\pi_{n-1}$ is in $\pi$ and it would not be a left-to-right maximum since $\pi_1>\pi_2$ would appear before it and would not be a right-to-left minimum since $\pi_{n-2}=3$ would appear to its right. Thus for similar reasons to above, we would have to have $\pi_{n-2}>\pi_2$ and there would only be one case: that $n$ is even and taking $r=n/2+1$, we have \[\pi = (n+1-j)(r-1)(r-2)\ldots (r-i+2)n(r-i+1)\ldots 32(n-1)(n-2)\ldots r21.\]
    Again it is straightforward to check these permutations are indeed shallow.
    Taken altogether, this implies that \[p_n = q_n + q_{n-2} + \bigg\lfloor \frac{n}{2}\bigg\rfloor-1.\]

    Next, let us consider those that have $\pi_1=n$. If $\pi_n=1$, then by Lemma~\ref{lem: n1}, there are $p_{n-2}$ such permutations. If $\pi_{n-1}=1$, then since $\pi$ is persymmetric, we must have $\pi_n=2.$ Then $\pi^{RL}$ is a persymmetric permutation that ends with 1, which  are also enumerated by $q_{n-2}$. Finally, by a similar proof to the one above, there is exactly one shallow 123-avoiding shallow permutation that starts with $n$ and has $\pi_i=1$ for $\lfloor \frac{n}{2}\rfloor + 1\leq i\leq n-2.$
    Now, this implies that \[q_n = p_{n-2} + q_{n-2} + \bigg\lfloor \frac{n-1}{2}\bigg\rfloor-1.\]
    Taking $P_{123}(x)$ and $Q_{123}(x)$ to be the respective generating functions. These recurrences together with the initial conditions imply that 
    \[P_{123}(x) = (1+x^2)Q_{123}(x) + \frac{x^5+x^4}{(1-x^2)^2} + 1+x^2\]
    and 
    \[Q_{123}(x) = x^2P_{123}(x) +x^2Q_{123}(x) + \frac{x^6+x^5}{(1-x^2)^2} +x.\]
    Solving for $P_{123}(x)$ gives us the generating function in the statement of the theorem.
\end{proof}

\section{Shallow permutations that avoid 321}\label{sec:321} 

Diaconis and Graham \cite{DG77} pointed out that permutations which satisfy the upper and lower bound of their inequality are enumerated by the bisection of the Fibonacci numbers, $F_{2n-1}$. These permutations were further discussed and characterized in \cite{HM13}. We start this section by providing an independent proof of this enumeration. We then enumerate these permutations by their descent count as well as those that exhibit certain symmetry.

\subsection{Enumeration of 321-avoiding shallow permutations}
\begin{theorem} \label{thm: 321}
    For $n \geq 1$, $t_n(321) = F_{2n-1}$, where $F_{2n-1}$ is the $(2n-1)$-st Fibonacci number. 
\end{theorem}

Before proving this theorem, we will prove the following lemma, which determines what these permutations must look like when $n$ occurs before position $n-1$. 

\begin{lemma} \label{lem: 321end}
    Let $n\geq 3$. If $\pi \in \mathcal{T}_n(321)$ has $\pi_j=n$ with $1 \leq j < n-1$ then:
    \begin{itemize}
        \item $\pi_n = n-1$
        \item $\pi^R \in \mathcal{T}_n(321)$
        \item $\pi_{k} = k-1$ for $j+2\leq k \leq n,$ 
    \end{itemize}
\end{lemma}

\begin{proof}
    Let $\pi\in\T_n(321)$ with $n\geq 3$. Since $\pi$ avoids $321$ we must have $\pi_{j+1} < \pi_{j+2} < \cdots < \pi_n$. By Theorem \ref{thm: shallowRecursive}, since $\pi$ is shallow, $\pi_n$ must be either a left-to-right maximum or right-to-left minimum in $\pi^R = \pi_1 \cdots \pi_{j-1} \pi_n \pi_{j+1} \cdots \pi_{n-1}$. It cannot  be a right-to-left minimum because $j<n-1$ and $\pi^R_{j+1} = \pi_{j+1} < \pi_n = \pi^R_j$. So $\pi_n$ must be a left-to-right maximum in $\pi^R$.  If $\pi_n \not= n-1$, since it is a left-to-right maximum in $\pi^R$, $n-1$ must occur after position $j$ in $\pi^R,$ and thus in $\pi$. 
    However, this means $\pi$ contains $n(n-1)\pi_n$ as a subsequence, which is a 321 pattern. Thus $\pi_n=n-1.$ This completes the proof of the first bullet point.

Note that the previous paragraph also implies that if $\pi \in \mathcal{T}_n(321)$ with $\pi_j = n$ where $1 \leq j < n-1$ then $\pi^R \in \mathcal{T}_n(321)$. Indeed, by Theorem \ref{thm: shallowRecursive} $\pi^R$ is still shallow and we form $\pi^R$ by replacing $n$ with $n-1$, so $\pi^R$ is still $321$ avoiding since $\pi$ was. This establishes the second bullet point.

We can combine the first two bullet points to prove the third. If $\pi \in \mathcal{T}_n(321)$ has $\pi_j = n$ with $1 \leq j < n-1$ then the first and second bullet point imply that $\pi^{R^m} \in \mathcal{T}_{n-m}(321)$ with $\pi_j = n-m$ for $1 \leq m \leq n-j-1$. When $1 \leq m \leq n-j-2$ we have $j \leq n-m-2$, in this case the first bullet point shows that $\pi_{n-m} = \pi^{R^m}_{n-m} = n-m-1$. This is equivalent to $\pi_k = k-1$ for $j+2 \leq k \leq n-1$ which in combination with the first bullet point proves the third.
\end{proof}

As an example, if we have a permutation $\pi\in\mathcal{T}_{13}(321)$ with the element 13 in position $8$, then we must have that the permutation $\pi$ ends with $\pi_{10}\pi_{11}\pi_{12}\pi_{13}=9(10)(11)(12)$. Note that $\pi_9$ is not determined by this lemma.


We are now able to prove Theorem~\ref{thm: 321}.

\begin{proof}[Proof of Theorem \ref{thm: 321}]
    Let $\pi \in \mathcal{T}_n(321)$. If $\pi_n = n$, by Theorem \ref{thm: shallowRecursive}, $\pi^R$ obtained by removing $n$ will be shallow and still $321$ avoiding. Similarly, we can append $n$ to the end of any $\tau \in \mathcal{T}_{n-1}(321)$ to obtain a permutation in $\mathcal{T}_n(321)$. Therefore, there are $t_{n-1}(321)$ permutations $\pi \in \mathcal{T}_n(321)$ with $\pi_n = n$. Similarly, if $\pi_{n-1}=n$, then $\pi^R$ is obtained by replacing $n$ with $\pi_n$, which is equivalent to deleting $n$ from $\pi$. One can clearly add $n$ into the $(n-1)$st position any $\pi\in\T_{n-1}(321)$ and obtain a permutation that is still shallow and 321-avoiding. This shows that there are $t_{n-1}(321)$ permutations $\pi \in \mathcal{T}_n(321)$ with $\pi_{n-1} = n$.

    Now let us see that there are $t_{j}(321)$ permutations $\pi \in \mathcal{T}_n(321)$ with $\pi_{j} = n$ for $1 \leq j \leq n-2$. Suppose $\pi\in\mathcal{T}_n(321)$ with $\pi_j=n$ and $1\leq j\leq n-2.$ A direct consequence of Lemma \ref{lem: 321end} is that $\pi^{R^{n-j}} \in \mathcal{T}_{j}(321)$. This is actually a bijection. Indeed, given any $\tau \in \mathcal{T}_{j}(321)$, we can form a new permutation $\hat{\tau} \in S_n$ with  $\hat{\tau}_m = \tau_m$ for $1 \leq m < j$, $\hat{\tau}_{j} = n$, $\hat{\tau}_{j+1} = \tau_{j}$, and $\hat{\tau}_{k} = k-1$ for $j+2 \leq k \leq n$. For example, given the permutation $\tau = 41263857 \in \T_8(321),$ we can obtain the permutation $\pi = 4126385(13)79(10)(11)(12)\in \mathcal{T}_{13}(321).$ It is clear that the permutation $\hat\tau$ formed is 321-avoiding, and it is shallow since $\hat\tau^{R^{n-j}}=\tau$ is.
    
    As this exhausts all the possible positions of $n$, we conclude that
    \[
    t_n(321) = 2 t_{n-1}(321) + \sum_{i=1}^{n-2} t_{i}(321)
    \]
    which, together with the initial conditions, is satisfied by $F_{2n-1}.$
\end{proof}

\subsection{321-avoiding shallow permutations by descent number}

In this subsection, we consider those shallow 321-avoiding permutations with $k$ descents.

\begin{theorem}\label{thm: 321-descents}
    Let $a_{n,k}$ be the number of permutations in $\T_n(321)$ with $k$ descents and let 
    $A(x,z) = \sum_{n,k} a_{n,k} x^k z^n$.  Then,
    \[
    A(x,z) = \frac{z - 2 z^2 + x z^2 + z^3 - x z^3}{1 - 3 z + 3 z^2 - 2 x z^2 - z^3 + x z^3}.
    \]
\end{theorem}

\begin{proof}
Let $a_{n,k}$ denote the number of permutations $\pi \in \T_n(321)$ with $k$ descents and $b_{n,k}$ the number of such permutations with $\pi_{n-1} = n$.

    Let $\pi \in \T_n(321)$ have $\pi_{n-1} = n$ and $k$ descents and consider the value of $\pi_{n-2}$. If $\pi_{n-2} = n-1$ then $\pi^R \in \S_{n-1}$ is still a shallow 321-avoiding permutation and has $\pi^R_{n-2} = n-1$. Since $\pi$ has $k$ descents, $\pi^R$ will also have $k$ descents. These are precisely the permutations enumerated by $b_{n-1,k}$. This construction is clearly reversible so there are $b_{n-1,k}$ permutations $\pi \in \T_n(321)$ with $k$ descents, $\pi_{n-1} = n$ and $\pi_{n-2} = n-1$.

    If $\pi_{n-2} \not= n-1$ this forces $\pi_{n-2} < \pi_{n}$, otherwise we have a $321$ consisting of $(n-1) \pi_{n-2} \pi_n$. This means $\pi^R$ will have one fewer descent, since we are removing the descent in position $n-1$. In other words, $\pi^R$ can be any  permutation $\pi' \in \T_{n-1}(321)$ with $k-1$ descents and $\pi'_{n-2} \not= n-1$. These are precisely enumerated by $a_{n-1,k-1} - b_{n-1,k-1}$. Again, this construction is reversible, so there are $a_{n-1,k-1} - b_{n-1,k-1}$ shallow 321-avoiding permutations of size $n$ with $k$ descents, $\pi_{n-1} = n$ and $\pi_{n-2} \not= n-1$.

    This implies the following recursion for $b_{n,k}$: 
    \[
b_{n,k} = b_{n-1,k} + a_{n-1,k-1} - b_{n-1,k-1}.
\]

Now, if $\pi \in \T_n(321)$  with $k$ descents and $\pi_n = n$, then $\pi^R \in \T_{n-1}(321)$ with $k$ descents. This is reversible, so there are $a_{n-1,k}$ such permutations.
    If $\pi_j = n$ with $1 \leq j \leq n-1$, then since $\pi$ is $321$ avoiding we must have $\pi_{j+1} < \pi_{j+2} < \cdots < \pi_{n}$. In order to have $k$ descents we therefore must have $k+1 \leq j \leq n-1$. We claim there are $b_{j+1,k}$ such permutations with $\pi_j = n$. This is clearly true when $j = n-1$ by construction. Now, if $k+1 \leq j \leq n-2$, by Lemma \ref{lem: 321end} since $\pi \in \mathcal{T}_n(321)$ has $\pi_{j} = n$ with $1 \leq j \leq n-2$, we have $\pi_{k} = k-1$ for $j+2 \leq k \leq n$.

    As a result, $\pi^{R^{n-j-1}} \in \S_{j+1}$ is a flat permutation with $k$ descents and $\pi^{R^{n-j-1}}_{j} = j+1$; these are precisely enumerated by $b_{j+1,k}$. Even stronger, thanks to Lemma \ref{lem: 321end}, reversing this operation produces all the $\pi \in \T_n(321)$ with $k$ descents and $\pi_j = n$. This proves the claim that such permutations are enumerated by $b_{j+1,k}$.

    Summing over all the possible positions of $n$ we find that
    \[
    a_{n,k} = a_{n-1,k} + \sum_{j=k+1}^{n-1} b_{j+1,k}. 
    \]
    Now, let $B(x,z) = \sum_{n,k} b_{n,k}x^kz^n$. These recursions together with the initial conditions imply that
    \[
    B(x,z) = z B(x,z) + xz A(x,z) - xz B(x,z) + z + z^2x.
    \]
    and
    \[
    A(x,z) = zA(x,z) + (B(x,z) - z - z^2x)(1 + z + z^2 + \cdots + z^{n-k-2}) + z.
    \]
    We need to remove $z$ and $z^2x$ because in the recursion we always have $n$ at least two greater than $k$, the number of descents. These are the only two terms in $B(x,z)$ where this does not occur. Furthermore, we can replace $1 + z + z^2 + \cdots + z^{n-k-2}$ by $\frac{1}{1-z}$ because the coefficients of $z^jx^k$ in $(B(x,z) - z - z^2x)$ for $0 \leq j \leq k+1$ are all zero. We can therefore conclude,
    \[
    A(x,z) = zA(x,z) + \frac{xzA(x,z)}{(1-z)(1-z+xz)}+ z.
    \]
    This gives us 
    \[
    A(x,z) = \frac{z}{1-z-\tfrac{xz}{(1-z)(1-z+xz)}}
    \]
    which simplifies to the desired generating function.
\end{proof}
Recall that Grassmannian permutations are those permutations with at most one descent. These permutations necessarily avoid 321 and thus we can obtain the following corollary. 
\begin{corollary}\label{thm:shallowgras}
    For $n\geq 2$, the number of shallow permutations of length $n$ that are Grassmannian is equal to $\binom{n+1}{3}+1.$
\end{corollary}

\begin{proof}
It follows from the generating function in Theorem~\ref{thm: 321-descents} that the generating function for shallow 321-avoiding permutations with exactly one descent is $\frac{\partial}{\partial x}\mid_{x=0} A(x,t) = \frac{z^2}{(1-z)^4}$ which tells us there are $\binom{n+1}{3}$ permutations in $\T_n(321)$ with 1 descent. Since there is 1 permutation of any size with zero descents and that permutation is shallow, the result follows.
\end{proof}

\subsection{321-avoiding shallow permutations with symmetry}

In this last subsection, let us consider those shallow 321-avoiding permutations that exhibit certain symmetry.
\begin{theorem}
    For $n\geq 1$, the number of involutions in $\mathcal{T}_n(321)$ is $F_{n+1}$, where $F_n$ is the $n$-th Fibonacci number. 
\end{theorem}

\begin{proof}
Let $i_n(321)$ be the number of 321-avoiding shallow involutions.  First we note that if $\pi\in \mathcal{T}_n(321)$ is an involution with $\pi_j=n$, then $j=n$ or $j=n-1$. To see this, consider $j<n-1$. Then by Lemma \ref{lem: 321end} we have $\pi_n=n-1$. But then since $\pi$ is an involution, we must have $\pi_{n-1}=n$, a contradiction. Therefore $n$ is in position $n$ or $n-1$.
It is clear that there are $i_{n-1}(321)$ such permutations that have $\pi_n=n$. Since any involution with $\pi_{n-1}=n$ must also have $\pi_n=n-1$, there are $i_{n-2}(321)$ permutations in $\mathcal{T}_n(321)$ with $\pi_{n-1}=n$.
With the initial conditions $i_1(321)=1$ and $i_2(321)=2$, the result follows. 
\end{proof}

\begin{theorem}
    The number of centrosymmetric 321-avoiding shallow permutations is $F_{n+1}$ when $n$ is even and $F_{n-2}$ when $n$ is odd, where $F_n$ is the $n$-th Fibonacci number.
\end{theorem}

\begin{proof}
    Let $c_n(321)$ be the number of shallow 321-avoiding centrosymmetric permutations and let us consider the position of $n.$ If $\pi_n=n$, then since $\pi$ is centrosymmetric, $\pi_1=1$. By removing both $n$ and $1$, we are left with a centrosymmetric shallow 321-avoiding permutation of size $n-2.$ Since this is reversible for any centrosymmetric permutation in $\T_n(321)$, there are $c_{n-2}(321)$ such permutations. If $\pi_{n-1}=n$, then we must have $\pi_2=1.$ In this case $\pi^{RL}$ is the same as the permutation obtained by deleting both 1 and $n$ (scaling as appropriate). The remaining permutation is a centrosymmetric shallow 321-avoiding permutation of size $n-2.$ Again, this is reversible for any centrosymmetric permutation in $\T_n(321)$, so there are $c_{n-2}(321)$ such permutations.

    Now consider the case where $\pi_{n-j}=n$ for $n-j\leq n-2$. Then since $\pi$ is centrosymmetric and satisfies Lemma~\ref{lem: 321end}, we must have \[\pi = 23\ldots j\pi_{j} 1 \ldots n\pi_{n-j+1}(n-j+1)(n-j+2)\ldots (n-1).\] Note that $\pi^{(RL)^{j}}$ leaves us with a centrosymmetric 321-avoiding shallow permutation of length $n-2j$. Thus, we get 
    \[c_n = c_{n-2}+c_{n-2}+c_{n-4}+c_{n-6}+\ldots\]
    which is equivalent to $c_n = 3c_{n-2}-c_{n-4}$ which together with the initial conditions is satisfied by $F_{n+1}$ when $n$ is even and $F_{n-2}$ when $n$ is odd. 
\end{proof}

\begin{theorem}
    The number of persymmetric 321-avoiding shallow permutations is $F_{n+1}$, where $F_n$ is the $n$-th Fibonacci number. 
\end{theorem}

\begin{proof}
    Let $p_n(321)$ be the number of shallow 321-avoiding persymmetric permutations and let us consider the position of $n.$ If $\pi_n=n$, then since $\pi$ is persymmetric, $\pi_1=1$. By removing both $n$ and $1$, we are left with a persymmetric shallow 321-avoiding permutation of size $n-2.$ Since this is reversible for any persymmetric permutation in $\T_n(321)$, there are $p_{n-2}(321)$ such permutations. If $\pi_{n-1}=n$, then we must have $\pi_1=2.$ In this case $\pi^{RL}$ is the same as the permutation obtained by deleting both 2 and $n$ (scaling as appropriate). The remaining permutation is a persymmetric shallow 321-avoiding permutation of size $n-2.$ Again, this is reversible for any persymmetric permutation in $\T_n(321)$, so there are $p_{n-2}(321)$ such permutations.

    Now consider the case where $\pi_{n-j}=n$ for $n-j\leq n-2$. Then since $\pi$ is persymmetric, we must have \[\pi = (j+1)12\ldots (j-1)\ldots n\pi_{n-j+1}(n-j+1)(n-j+2)\ldots (n-1).\] Note that $\pi^{(RL)^{j}}$ leaves us with a persymmetric 321-avoiding shallow permutation of length $n-2j$. Thus, we get 
    \[p_n = p_{n-2}+p_{n-2}+p_{n-4}+p_{n-6}+\ldots\]
    which is equivalent to $p_n = 3p_{n-2}-p_{n-4}$ which together with the initial conditions is satisfied by the Fibonacci numbers. 
\end{proof}

\section{Future directions and concluding remarks}\label{sec:conclude}

Theorems~\ref{thm: 132} and \ref{thm: 321} imply that $t_n(132)=t_n(321)$, since both are equal to $F_{2n-1}.$ In our paper, we prove these separately and directly, but it does raise the following question.

\begin{question}
    Is there a bijective proof that $t_n(132)=t_n(321)$?
\end{question}
Based on the numerical data, we can conjecture something stronger. We conjecture that there is a bijection 
$f: \T_n(132) \to \T_n(321)$ 
with the property that $\cyc(\pi) = \cyc(f(\pi))$ and $\des(\pi) + 1=\lrmax(f(\pi)),$ where $\lrmax(\sigma)$ is the number of left-to-right maxima in a permutation $\sigma.$ It seems likely that there are more statistics that could be preserved in bijections between shallow 132-avoiding and shallow 321-avoiding permutations. 

It may even be the case that this relationship between $\T_n(132)$ and $\T_n(321)$ goes deeper and could imply more interesting things about 132-avoiding shallow permutations. For example, it is known that the 321-avoiding shallow permutations have many nice properties (see \cite{E07,PT15}, among others): they have unimodal cycles, they avoid the patterns 321 and 3412, they satisfy both the upper and lower bound of the Diaconis-Graham inequality, etc. 
\begin{question}
    Are there any interesting characterizations of 132-avoiding shallow permutations that are in the same vein as those listed above for 321-avoiding shallow permutations?
\end{question}

Another possibility for future work is related to Theorem~\ref{thm: 3n12}. In that theorem we show that shallow permutations avoid certain mesh patterns. However, this is not a complete characterization of shallow permutations. This leads us to ask the following question
\begin{question}
    Can shallow permutations be characterized completely in terms of mesh pattern avoidance? That is, is there a set of mesh patterns $S$ so that $\pi$ is shallow if and only if it avoids all patterns in $S$?
\end{question}


Finally, there are many other questions about pattern-avoiding shallow permutations that we did not consider in this paper. In many cases, it seems reasonable to count these permutations by various statistics, like number of cycles. One might also consider shallow permutations that avoid longer patterns or sets of patterns. As a more general question, one could attempt to count pattern-avoiding permutations that are not shallow, but perhaps satisfiy $I(\pi) + T(\pi) +k=D(\pi)$ for a fixed value $k$, or perhaps pattern-avoiding permutations whose cycle diagrams correspond to different knots/links (see for example, \cite{CM24, W22}).

\end{document}